\documentclass[preprint, EJP]{ejpecp} % replace ECP by EJP if needed.
\usepackage[utf8]{inputenc}
\usepackage[T1]{fontenc}
\usepackage[english]{babel}

\usepackage{hyperref}
\usepackage{epsfig}
\usepackage{mathtools}

\usepackage{xcolor}

\usepackage{amssymb}
\usepackage{amscd}
\usepackage{tikz}
\usepackage{graphicx}
\usepackage{color}
\usepackage{verbatim}
\usepackage{pgf,tikz}
\usepackage{mathrsfs}
\usetikzlibrary{arrows}
\usepackage{todonotes}

\usepackage{ulem}
\mathtoolsset{showonlyrefs}
\SHORTTITLE{Large deviations for the almost-critical \ER~random graph}

\TITLE{An alternative approach to large deviations for the almost-critical \ER~random graph} % \thanks is optional. Insert line breaks with \\

\AUTHORS{%
  Luisa~Andreis\footnote{Politecnico di Milano, Italy.
    \EMAIL{luisa.andreis@polimi.it} \url{https://sites.google.com/view/luisaandreis/}} \orcid{0000-0003-2564-2635}
  \and %% remove this line and below if single author
  Gianmarco~Bet\footnote{Università degli Studi di Firenze, Italy. \EMAIL{gianmarco.bet@unifi.it} \url{https://people.dimai.unifi.it/bet/}} \orcid{0000-0001-8431-0636}
  \and
Maxence~Phalempin\footnote{Università degli Studi di Firenze, Italy. \EMAIL{maxence.phalempin@unifi.it}\url{https://maxence-phalempin.github.io/home-page/}}
  }%AUTHORS
 % Separate items with ;

\AMSSUBJ{60C05 ; % Combinatorial probability
05C80 % Random graphs (graph-theoretic aspects) 
} % Edit. Separate items with ;
\AMSSUBJSECONDARY{60F10% Large deviations
} % Optional, separate items with ;

\SUBMITTED{January xx, 2024} % Edit.
\ACCEPTED{} % Edit.

\VOLUME{0}
\YEAR{2023}
\PAPERNUM{0}
\DOI{10.1214/YY-TN}

\ABSTRACT{We study the near-critical behavior of the sparse Erd\H{o}s-R\'enyi random graph $\mathcal{G}(n,p)$ on $n\gg1$ vertices, where the connection probability $p$ satisfies $np = 1+\theta(b_n^2/n)^{1/3}$, with $n^{3/10}\ll  {b_n}\ll n^{1/2}$, and $\theta\in\mathbb{R}$. To this end, we introduce an empirical measure that describes connected components of $\mathcal{G}(n,p)$ of mesoscopic size $\propto (nb_n)^{2/3}$, and we characterize its large deviation behavior.  The proof hinges on detailed   combinatorial estimates and optimization procedures.  In particular, we give precise estimates for the probability that the graph has no connected component of mesoscopic size or larger. We argue that these are a stepping stone for the analysis of more general inhomogeneous random graphs. 
Our proof strategy gives new and accurate estimates of the probability that the sparse Erd\H{o}s-R\'enyi graph is connected.
}
\newcommand{\Ccal}{\mathcal C}
\newcommand{\Me}{\mathrm Me}

\newcommand{\PP}{\mathbb{P}}
\newcommand{\Pp}{\mathrm P}
\newcommand{\B}{\mathcal B}

\newcommand{\Mspace}{\M_{\N}(0,\infty)}

\newcommand{\Z}{\mathbb{Z}}
\newcommand{\N}{\mathbb{N}}
\newcommand{\M}{\mathcal{M}}

\newcommand{\Gcal}{\mathcal{G}}

\newcommand{\Ncal}{\mathcal{N}}

\newcommand{\R}{\mathbb{R}}

\newcommand{\ER}{Erd\H{o}s-R\'enyi}

\begin{document}

%%%%%%%%%%%%%%%%%%%%%%%%%%%%%%%%%%%%%%%%%%%%%%%%%%%%%%%%%%%%%%%%%%%
%%                                                               %%
%% No need for \maketitle.                                       %%
%%                                                               %%
%%%%%%%%%%%%%%%%%%%%%%%%%%%%%%%%%%%%%%%%%%%%%%%%%%%%%%%%%%%%%%%%%%%

%%%%%%%%%%%%%%%%%%%%%%%%%%%%%%%%%%%%%%%%%%%%%%%%%%%%%%%%%%%%%%%%%%%
%%                                                               %%
%% Please replace what follows by the body of your article       %%
%% (up to the bibliography):                                     %%
%%                                                               %%
%%%%%%%%%%%%%%%%%%%%%%%%%%%%%%%%%%%%%%%%%%%%%%%%%%%%%%%%%%%%%%%%%%%

\section{Introduction}

The \ER~(ER) graph with vertex set $[n]:=\{1,\dots,n\}$ is a random graph obtained by connecting any two vertices with probability $p_n\in(0,1)$, independently from all other connections. The resulting graph is said to be \textit{sparse} if $p_n=\omega_n/n$ with $\omega_n \rightarrow \omega\in\R^+$ as $n\to\infty$. It is by now well-known that the sparse ER graph undergoes a phase transition as $\omega$ crosses the critical threshold $\omega=1$, see, e.g., \cite{ER60,ER61}. This phase transition involves the sizes of the connected components of the graph (so-called \textit{clusters}) and it can be summarized as follows: when $\omega<1$, with high probability there is no cluster of size larger than $O(\log(n))$. Clusters of this size are also known as \textit{microscopic} components. When $\omega>1$, with high probability there is a unique connected component of size $O(n)$, also known as the \textit{macroscopic} (or giant) component, and all the remaining components are not larger than $O(\log(n))$.  The central limit theorem for the size of the largest component was proved in \cite{BaBoDe00,Mar98,Pit90}. The first large deviation result in this context can be found in~\cite{OCon98}. Later \cite{Pu05} proved large, moderate and normal deviations for the sequence of the macroscopic component(s) using the so called \textit{method of stochastic processes}, while large deviations for microscopic components are studied in~\cite{BoCa15}. The work \cite{LA21} unifies the two regimes within the framework of large deviations. More precisely, \cite{LA21} 
proves a large deviation principle (LDP) jointly for two empirical measures describing the distribution of, respectively, the macroscopic components and the microcoscopic components. As $\omega$ crosses the critical 
threshold $\omega=1$, the minimizer of the rate function abruptly changes from a configuration with no macroscopic components to a configuration with a unique macroscopic component. 
%The result in \cite{LA21} applies to the critical regime $\omega=1$ because there is no typical macroscopic behaviour, and hence the large deviation framework is not applicable. 
It is known that, when $\omega=1$, the sizes of the largest components are of order $O(n^{2/3})$. For this reason, these are also known as \textit{mesoscopic} components. In particular, \cite{AL97} shows that the phase transition occurs for all values of $p_n$ inside the critical window $\omega_n(\theta)=1+\theta n^{-1/3}$ for $\theta\in\R$. More precisely, \cite{AL97} shows that the vector of connected component sizes, rescaled by $n^{2/3}$, converges in law to the length of excursions of reflected Brownian motion with negative parabolic drift that depends on $\theta$. It follows that, in the critical window, component sizes do not concentrate around a typical deterministic value, and mesoscopic fluctuations of order $O(n^{2/3})$ occur with positive probability. In~\cite{Pu05}, Puhalskii further clarifies this picture by analysing the atypical behaviour of cluster size \textit{just outside} the critical window. Our work is closely related to~\cite{Pu05} as we also examine the fluctuations of component sizes in this regime; see Section \ref{sectopomn} for a detailed comparison. The novelty of our contribution lies in our approach. In fact, the technique we develop is rather general and is well-suited for the large deviation analysis of more general graphs, such as \textit{inhomogeneous} random graphs, as we now argue.

With the increasing availability of data on complex systems in the last two decades, there has been a growing interest in generalisations of the ER random graph that are better suited to modelling realistic networks. For a detailed introduction to the field, see \cite{vdHof16}. One of the most natural extensions of the ER graph is the \textit{inhomogeneous random graph} (IRG), first introduced in~\cite{Sod03}. The IRG is obtained by assigning to each vertex a \textit{type} (or weight) chosen in some \textit{type space} and then connecting any two vertices with a probability that depends only on their type, independently from all other connections. In~\cite{BoJaRi07} the authors prove that, under appropriate assumptions, the component sizes of the IRG undergo an ER-like phase transition; see also \cite{bet2020big,BhvdHoLe10} for analogous results for other versions of the IRG. In \cite{AnKoLaPa23} the authors extend the combinatorial approach first developed in~\cite{LA21} for the ER random graph to the IRG in order to prove a large deviation principle for the empirical measures of the microscopic and macroscopic components \textit{outside} of the critical window. In a similar fashion, our result for the near-critical fluctuations of the ER random graph is a crucial step towards the analysis of more general models such as the IRG. 
% However we believe that approaching the inhomogeneous graph with the \textit{method of stochastic processes} from~\cite{Pu05} would be very hard, if not impossible. 
% For this reason, with the extension to the inhomogenous graph in mind, we provide here a proof of the fluctuations of the size of mesoscopic components right outside of the critical window in the ER graph.  
%
While writing this paper we learned of the recent work~\cite{Sun23}, which gives alternative proofs of \textit{other} results from~\cite{Pu05}, by rewriting relevant quantities in terms of Poisson point processes. In future research, it would be interesting to compare the approach in~\cite{Sun23} to ours.
%We leave to future work the comparison between this approach and ours.

Our main result is an LDP for the largest connected components of an ER graph just outside of the critical window. Our proof is inspired by the approach in~\cite{LA21} and relies on careful combinatorial estimates. It is based on the following three key arguments. First, the probability mass function of the ER graph can be expressed as the product of terms involving the contribution of each connected component. This allows us to identify the influence of microscopic, mesoscopic and macroscopic connected components. Second, the distribution of the microscopic components of maximal mass satisfies a characteristic equation. By studying this equation, we are able to quantify the influence of microscopic components on the total mass. Third, through careful combinatorial analysis, we are able to obtain a precise approximation of the probability that an ER graph in a specific (sparse) regime is fully connected. To the best of our knowledge, this result is novel and of independent interest. Moreover, it is enough to obtain an estimate of the probability mass of each mesoscopic component. The three steps outlined here form the basis of our approach. In our main theorem we prove that these steps are sufficient to obtain the large deviations principle.

A large deviation result similar to ours was proved with considerably different techniques in \cite{Pu05}. The strength of our contribution lies in our proof technique which we sketched above. We believe that this approach could be successfully adapted to obtain large deviations principles for other sparse graphs close to criticality, in particular for the inhomogeneous random graph. 
%, while we see difficulties in extending the \textit{method of stochastic processes} to the inhomogeneous case. 
Nevertheless, it is important to note that our proofs require that $p_n=(1 + \theta (b_n^{2}/n)^{1/3})/n$ is not too close to the critical window (formally, $b_n\gg n^{3/10}$), while the proof in~\cite{Pu05} works for any sequence $p_n$ arbitrarily close to (but still outside of) the critical window. However, our assumption regarding $p_n$ is purely technical. We conjecture that it can be eliminated entirely with even more precise combinatorial estimates, although we chose not to pursue this.
%
%We do not see this as a limitation of our approach, since this is a technical condition that we needed to get some specific bounds, which we are confident can be improved further. 
% We stress that the importance of our work is indeed to set a new way to approach the problem of large deviations for sparse graphs near-criticality.

% A result similar to ours was proved with considerably different techniques in \cite{Pu05}. The novelty of our contribution lies in our approach. In fact, our technique is rather general and is well-suited for the large deviation analysis of more general graphs, such as inhomogeneous random graphs. 
%For the critical regime and the behaviour of the microscopic and macroscopic components of these graphs, see \cite{LUWO23}.
% In \cite{AnKoLaPa23},  the first two key facts above are settled for  the inhomogeneous random graph.  Thus, we believe a similar strategy to ours could be followed once a similar estimate as the third key argument is made for multitypes graphs.    

The rest of the paper is organized as follows. In Section \ref{secpreliminaries} we present the setting and notation as well as the main theorem. We also give an overview of the proof and we state some intermediate results of independent interest, such as an estimate of the probability that an ER graph of critical size is fully connected and an estimate of the contribution of the microscopic components to the total probability mass. These will be proven, respectively, in Section \ref{Sec:prob_conn} and Section \ref{sectionpetitecomponent}. Section \ref{sec:proof_main_theorem} is dedicated to the proof of the main result. In Section \ref{appendix} we carry out the detailed proof of some technical results, and we give additional details on the topology we consider.

\section{Preliminaries and main result}\label{secpreliminaries}
 We denote by $\mathbb R^*:=\mathbb R \backslash \{0\}$ and $\bar{{\mathbb R}}_+:=\mathbb R_+ \cup \{\infty\}$. Given two eventually non zero 
 sequences $a,b\in \mathbb R ^\N$, the notation $b\preceq a$ means $\liminf_{n \to \infty} \frac{a_n}{b_n}>1$. 

Let $(b_n)_{n \in \N}\in \mathbb{R}^{\N}$ be an increasing sequence satisfying $b_n=o(\sqrt{n})$ and let 
\begin{equation}\label{eq:def_p}
p_n(\theta):=\frac{1+\theta \left(\frac{b_n^2}n\right)^{1/3}+O\left(\left(\frac{b_n^2}n\right)^{2/3}\right)}{n}, \qquad \theta\in\R.
\end{equation} 
%
%We see that this means being right outside of the critical window,
Note that for any choice $\theta\in\R$, $p_n(\theta)\in(0,1)$ for $n$ large enough. We will consider the \ER~random graph $\mathcal{G}(n,p_n(\theta))$. Denote by $\mathbb{P}_{n,p_n(\theta)}$ the probability law induced by $\mathcal{G}(n,p_n(\theta))$ on the space of simple graphs with $n$ vertices. The edge probability $p_n(\theta)$ lies outside, but arbitrarily close to, the critical window of the ER graph from \cite{AL97}.

Let us consider the sequence of random variables $L^n:=(L^n_k)_{k=1,\ldots,n}$, where 
\begin{align*}%
L_k^n:=\vert\{\mathcal C : \mathcal C\text{ is a connected component of }\mathcal{G}(n,p_n(\theta))\text{ of size } k\}\vert,
\end{align*}%
with $\vert A \vert$ denoting the cardinality of the set $A$. Note that $L^n\in\Ncal_n$, where $\Ncal_n$ is the space of sequences 
\begin{align*}
\Ncal_n:=\{l=(l_k)_{k=1}^n\in\N^n : \sum_{k=1}^nkl_k=n \}.
\end{align*}
In order to understand the behavior of the connected components of the graph, we define a measure that describes components of {mesoscopic} size. The \textit{mesoscopic measure} is the following measure: 
%\gm{We need to pick a style for fractions in pedices: either $\frac{a}{b}$, or $a/b$, and then change it throughout the paper. I prefer the latter}\lu{I agree with Gianmarco, let's uniform to $a/b$!}
%
\begin{align*}%
\Me_n(L^n):=\sum_{k\in \N}L^n_k\delta_{k/(nb_n)^{2/3}}=\sum_{\Ccal\text{ connected component}}\delta_{|\Ccal|/(nb_n)^{2/3}}.
\end{align*}%
Formally, $\Me_n$ is an element of the space of integer-valued measures with locally finite weights, i.e.
\begin{align*}
\M_{\N}(0,\infty):=\{\mu \in \M_{\N}, \int fd\mu <\infty, \forall f\in C^0_c((0,\infty),\mathbb{R})\},
\end{align*}
where $\M_{\N}$ is the space of integer-valued measures on $(0,\infty)$ and $C_c((0,\infty)$ is the space of continuous real-valued functions with compact support in $(0,\infty)$.
We endow $\M_\N(0,\infty)$ with the vague topology.
%To each space of measure, we associate the topology given by the convergence in law and fix the following equivalent metric :\\
Note that any measure $\mu \in \M_{\N}(0,\infty)$ may be written as $\mu:=\sum_{i\in \Z}\delta_{u_i}$, with the convention $1>u_{-1}\geq u_{-2}\geq \dots>0$ and $1\leq u_0 \leq u_1 \leq \dots$ and completed into a full sequence in $\overline{\mathbb{R}}_+^\Z$ by appending elements $0$ and $\infty$. We adopt this notation from now on.  
We are now able to state our main result.
\begin{theorem}\label{thmweakldp}
The law of the random variable $\Me_n(L^n)$ under $\mathbb{P}_{n,p_n(\theta)}$ satisfies a \textit{large deviation principle} with speed $b_n^2$ and rate function $I :\M_\N(0,\infty) \mapsto \bar{\mathbb{R}}_+$ defined by
%\label{eq:rate_function}
\begin{equation*}
    I(\mu):=-\frac{1}{24}\int_0^\infty x^3d\mu+\frac{1}{6}(\int_0^\infty xd\mu-\theta)^31_{(\int_0^\infty xd\mu \geq \theta)}+\frac{\theta^3}{6},    
\end{equation*}
%\sum_{i\in \Z} u_i^3+\frac{1}{6}(\sum_{i\in\Z} u_i-\theta)^31_{(\sum_{i\in\Z} u_i\geq \theta)}+\frac{\theta^3}{6}
whenever the integral $\int_0^\infty xd\mu$ is finite and $I(\mu)=\infty$ otherwise. 
\end{theorem}
As the result of the regularity of $I(\cdot)$ on $\M_\N(0,\infty)$ as a good rate function discussed in section \ref{sectopomn}, the above theorem is equivalent to the following statement :
\begin{align*}
\lim_{\delta\to 0}\liminf_{n\to\infty}&\frac{1}{b_n^2}\log \PP_{n,p_n(\theta)}\left(\Me_n(L^n)\in B_{\delta}(\mu)\right)\geq -I(\mu),\\
    \lim_{\delta\to 0}\limsup_{n\to\infty}&\frac{1}{b_n^2}\log \PP_{n,p_n(\theta)}\left(\Me_n(L^n)\in B_{\delta}(\mu)\right)\leq -I(\mu),
\end{align*}
where $B_{\delta}(\mu)$ is a ball of radius $\delta$ around $\mu$ in the vague topology. 

\subsection{The starting point: explicit expression and combinatorics}
We make the following additional assumption.
\begin{hypothesis}\label{hypbn}
    The sequence $(b_n)_{n\in \N}$ satisfies
    \begin{align*}
      n^{3/10}\ll  {b_n}\ll n^{1/2}.
    \end{align*}
\end{hypothesis}
The lower bound in the above assumption implies that our result holds only when $p_n$ is sufficiently distant from the critical window. This is in contrast to the corresponding result in~\cite{Pu05}, which holds for any sequence $p_n$ outside of the critical window. However, Hypothesis \ref{hypbn} is purely technical, and we conjecture that it can be improved, or possibly even removed, with even more precise optimizations. Since this is not the central focus of this work, we choose not to pursue this.

The starting point of our approach is the explicit expression for the probability of observing a given graph under the law $\PP_{n,p}(\cdot)$ of the ER graph $\mathcal{G}(n,p)$. In particular, fixing $l \in \mathcal{N}_n$, we have the following (almost) explicit expression for the probability of the event  $\{L^n=l\}$: %With an abuse of notation, if we denote by   $\PP_{n,p}$ the law of the ER graph $\mathcal{G}(n,p)$ (and we will continue using this notation whenever there is no confusion with $\PP_{n,\theta}$), 
 
%Let us use known result to say that, for any fixed $l\in\Ncal_n$
%
\begin{align}\label{eqprimord}
\PP_{n,p}(L^n=l)=n!\prod_{k=1}^n\frac {\PP_{k,p}(L^k_k=1)^{l_k}(1-p)^{\frac12(n-k)kl_k}}{l_k!(k!)^{l_k}}.
\end{align}

The expression of $\PP_{n,p}(L^n=l)$ depends on $\PP_{k,p}(L^k_k=1)$, namely the probability that $\Gcal(k,p)$, a smaller graph, is connected. We aim to exploit this expression to prove Theorem~\ref{thmweakldp}. One crucial element of the proof relies on the product form of \eqref{eqprimord} . We divide the product into three terms, one that collects the contributions of small values of $k$ (formally $k\ll (nb_n)^{2/3}$), one of mesoscopic values $k\propto (nb_n)^{2/3}$ and one of large values $k\gg (nb_n)^{2/3}$. A similar strategy has been used for $p$ outside of the critical window in \cite{LA21}, in order to get a  LDP for the component sizes of $\Gcal(n,\frac{t}n)$, for $t>0$. There, the authors needed precise estimates for the probability $\PP_{k,p}(L^k_k=1)$ in different regimes of $k$ and $p$ and they used results from \cite{Ste70}. In our case, precise asympotics for the probability $\PP_{k,p}(L^k_k=1)$  when $k\propto (nb_n)^{2/3}$ and $p=p_n(\theta)$ are necessary. To our knowledge, these estimates have not been proved previously in the literature, and we derive them in Section~\ref{Sec:prob_conn}. 

Another crucial element in the proof of Theorem~\ref{thmweakldp} is an accurate estimate, with error at most exponentially small, of the probability that there is no component of size $\propto (nb_n)^{2/3}$ and larger. We obtain this with the use of a special series, which plays a crucial role for the graph, details are given in Section~\ref{sectionpetitecomponent}.

\subsection{Proof strategy: splitting contribution of microscopic, mesoscopic and macroscopic components}

%In order to state a LDP for microscopic components on Erdos-Reniy graphs, we introduce the following metrics : To an observation of an element $l$ in the Erdos Renyi graph $\mathcal{G}(n,p)$, one may define respectively the mesoscopic measure $Mi_n(l):=\frac 1 n \sum_{k\in \N^*}l_k\delta_k$, the macroscopic measure $Ma_n(l):=\sum_{k\in \N^*}l_k\delta_{\frac{k}{n}}\in \M_{\N}(0,1)$ and the mesoscopic measure\\ $Me_n(l):=\sum_{k\in \N^*}l_k\delta_{\frac{k}{(nb_n)^{2/3}}}\in \M_{\N}(0,\infty)$ where\\ $\M_{\N}(0,\infty):=\{\mu \in \M_{\N}, \int fd\mu <\infty, \forall f\in C^0_c((0,\infty),\mathbb{R})\}$ is a (non compact) set of integer valued measures with locally finite weights.To each space of measure, we associate the topology given by the convergence in law and fix the following equivalent metric :\\Any measure $\mu,\nu \in \M_{\N}(0,\infty)$ may be written as $\mu:=\sum_{i\in \Z}\delta_{u_i}$ and $\nu:=\sum_{i \in \Z}\delta_{v_i}$ with $u,v\in \mathbb{R}^{\Z}$ two sequences satisfying the following conventions $1>u_{-1}\geq u_{-2}\geq \dots>0$ and $1\leq u_0 \leq u_1 \leq \dots$ and completed into full sequence by appending elements $0$ and $\infty$.The mesoscopic distance $dm$ is then defined by \begin{align*}dm(\mu,\nu)=\min\{du,v),du,(v_{i+1})_{i\in \Z}),d(u_{i+1})_{i\in \Z},v)\}.\end{align*}with $d$ the distance defined for any $u,v \in \left(\mathbb R_+^*\right)^{\Z}$ by $$du,v)=\sum_{i\leq 0}\frac{1}{2^{|i|}}|u_i-v_i|+\sum_{i\geq 0}\frac{1}{2^{|i|}}|e^{-u_i}-e^{-v_i}|$$
 %\begin{proof}[Proof of Theorem~\ref{thmweakldp}]

The first crucial step in the proof is an almost explicit decomposition of the law of $L^n$ under the probability measure $\PP_{n,p}$. Let
\begin{align}\label{eq:z_k}
    z_k^n(l_k):=\frac {\PP_{k,p}(L^k_k=1)^{l_k}(1-p)^{\frac12(n-k)kl_k}}{l_k!(k!)^{l_k}}.
\end{align}
For any $l\in \mathcal{N}_n$, we first split  \eqref{eqprimord} between the contribution of \textit{microscopic} components of size smaller than $\alpha_n:=\epsilon (nb_n)^{2/3}$ those of \textit{macroscopic} of size greater than $\beta_n:=\lceil n^{17/24}b_n^{8/12}\rceil$ and the remaining \textit{mesoscopic} components. Notice that, while the  threshold $\alpha_n$ is the most natural one to distinguish between microscopic and mesoscopic components, the threshold $\beta_n$ is somehow arbitrary and it does not have a fine interpretation in terms of the graph (any $\beta_n \ll n^{3/4}b_n^{1/2}$ would work). In particular, it is large enough that the estimates for macroscopic components (the ones with size larger than $\beta_n$) work smoothly, but small enough that the same is true for the estimates for the mesoscopic components (the ones with size larger than $\alpha_n$ and smaller than $\beta_n$). This let us rewrite the probability in the following way:

%\lu{comment on why} {\color{blue}(the choice of $\beta_n$ is arbitrary so long as it satisfies $\beta_n \ll n^{3/4}b_n^{1/2}$ which is the upper bound of the estimates in proposition \ref{lemgraphconnex})} and the remaining \textit{mesoscopic} components.
%We actually split further the equation \eqref{eqprimord0}, distinguishing also the contribution of \textit{macroscopic} components of size $k\geq \beta_n$ with $\beta_n$ set as $\beta_n=\lceil n^{17/24}b_n^{8/12}\rceil$.\gm{Why do we do the split only now? Also, in the introduction we defined macroscopic components as those with $k\gg \alpha_n$. This should be the same} 
%
\begin{align}%\label{eqprimord}
\PP_{n,p}(L^n=l)&=n!\prod_{k=1}^nz_k^n(l_k)\nonumber\\
&=n!\prod_{k=\alpha_n}^nz_k^n(l_k)\prod_{k=1}^{\alpha_n-1}z_k^n(l_k)\nonumber\\
&=e^{\sum_{k=\alpha_n}^n\log(l_k!)}n!\prod_{k=\beta_n}^nz_k^n(1)^{l_k}\prod_{k=\alpha_n}^{\beta_n}z_k^n(1)^{l_k}\prod_{k=1}^{\alpha_n-1}z_k^n(l_k)\label{eqprimord0}.
\end{align}%
%
%\lu{here the dependence on $n$ of $z_k^n(l_k)$ is hidden (e.g. $\PP_{k,p}(L^k_k=1)$ depends non trivially on $n$). Please, either make it much clearer or avoid writing something like\[z_k^{n-\sum_{k=\alpha_n}^{\beta_n} kl_k}(l_k)\]as you write below in the definition of $F_{Mi}$. It is not clear at all what $z_k^{n-\sum_{k=\alpha_n}^{\beta_n} kl_k}(l_k)$ means.  I suggest simply to write $z_k^{n}(l_k)$ in the definition of $F_{Mi}$ and then explain that $F_{Mi}$  corresponds to the probability of a smaller graph to have no component larger than something. I've started changing some notation below, there are terms like $(1-p)^{1/2(\sum il_i)kl_k}$ to be handled.}
%Notice that $\sum_{k=m_n(nb_n)^{2/3}}^n\log(l_k!)\leq \sum_{k=m_n(nb_n)^{2/3}}^nl_k\log(l_k)=O\left(\frac{n^{1/3}}{m_nb_n^{2/3}}\log(n)\right)=o(b_n^2)$ \marginnotemx{whenever $m_n\geq ...$}.

Recall that $\sum_{k=1}^n kl_k=n$. We immediately see that Hypothesis \ref{hypbn} implies that  $\sum_{k\geq \alpha_n}l_k\leq \frac{n^{1/3}}{\epsilon b_n^{2/3}} = o(b_n^2)$ and that $\sum_{k\geq \alpha_n}\log(l_k!) = o(b_n^2)$ for all $l\in\Ncal_n$. Using Stirling formula for factorial terms, we see that the following holds as well
\begin{align*}
   \frac{n!}{(n-\sum_{k=\alpha_n}^nkl_k)!} &=e^{o(b_n^2)} \left(\frac{n-\sum_{k=\beta_n}^nkl_k}{e}\right)^{\sum_{k=\alpha_n}^{\beta_n}kl_k} \\
   &\quad\times\left(\frac{n-\sum_{k=\beta_n}^nkl_k}{n-\sum_{k=\alpha_n}^{n}kl_k}\right)^{n-\sum_{k=\alpha_n}^{n}kl_k} n^{\sum_{k\geq \beta_n}kl_k}.\notag
\end{align*}
Summarizing from equation \eqref{eqprimord0}, for any $l\in\Ncal_n$ we have
\begin{align}\label{eqprimord2}
\PP_{n,p}(L^n=l)&=e^{o(b_n^2)}F_{Mi}(l)F_{Me}(l)F_{Ma}(l),
\end{align}
with $F_{Mi}(l)$, $F_{Me}(l)$ and $F_{Ma}(l)$ defined respectively as
%&=e^{o(b_n^2)}\left(\frac{n}{n-\sum_{k=\alpha_n}^nkl_k}\right)^{n-\sum_{k=\alpha_n}^{\beta_n}kl_k}\prod_{k=\alpha_n}^\beta_n n^{kl_k}(z_{k}^n(1))^{l_k}(1-p)^{(n-\frac{1}{2}(\sum_{i=\alpha_n}^n il_i+k))kl_k}\label{eqpmesocomp}\\
%&(n-\sum_{k=m_n(nb_n)^{2/3}}^nkl_k)!\prod_{k=1}^{m_n(nb_n)^{2/3}-1}z_k^{n-\sum_{k=\alpha_n}^{\beta_n} kl_k}(l_k)\label{eqpminorcomp},
%
\begin{align}
 F_{Mi}(l)  &=  (n-\sum_{k=\alpha_n}^nkl_k)!(1-p)^{-\frac 12(\sum_{k=1}^{\alpha_n}kl_k)(\sum_{k=\alpha_n}^{n}kl_k)}\prod_{k=1}^{\alpha_n-1}z_k^{n
 %-\sum_{k=\alpha_n}^{\beta_n} kl_k
 }(l_k);\label{eq:F_mi}\\
 &\nonumber\\
 F_{Me}(l)  &= \left(\frac{n-\sum_{k=\beta_n}^nkl_k}{n-\sum_{k=\alpha_n}^nkl_k}\right)^{n-\sum_{k=\alpha_n}^nkl_k}\label{eq:F_me}\\
 & \quad\prod_{k=\alpha_n}^{\beta_n} \left(\frac{n-\sum_{k=\beta_n}^nkl_k}{e}(1-p)^{\frac{1}{2}(\sum_{i=1}^{\alpha_n}il_i)}\right)^{kl_k}(z_{k}^n(1))^{l_k} ;\notag\\
 &\nonumber\\
F_{Ma}(l)&=\prod_{k=\alpha_n}^n n^{kl_k}(z_{k}^n(1))^{l_k}(1-p)^{\frac{1}{2}(\sum_{i=1}^{\alpha_n} il_i)kl_k}.\label{eq:F_ma}
\end{align}

%Notice that the term $F_{Mi}(l)$ corresponds to the probability $\PP_{N,p}(L^N=l_-)$ with $N:=n-\sum_{k=\alpha_n}^nkl_k$ and $l_-$ are the first $\alpha_n$ coefficients of $l$.

In the above decomposition, we clearly see the contribution to the probability of, respectively, microscopic, mesoscopic and macroscopic components. By abuse of notation, we are omitting the dependence on $n$ and on $p$ of the three above terms. 
 Notice that we manipulated the combinatorial factors in a way that the term $F_{Mi}(l)$ in \eqref{eq:F_mi} corresponds to the probability $\PP_{N,p}(L^N=\hat{l})$, with $N\colon=n-\sum_{i=\alpha_n}^{n}il_i$ and $\hat{l}$ being a $N$-dimensional vector whose first $\alpha_n-1$ components correspond to the ones of $l$ and the remaining are zeros.  Roughly speaking, we will see that the probability mass of an event of the form
\begin{align*}%
\{\Me_n\approx \mu\},\qquad \mu\in \Mspace,
\end{align*}%
is dominated by the term $F_{Me}$ and  $F_{Mi}$. The other term  $F_{Ma}$ will be proved to be of lower 
order, and it can be bounded independently of $\mu$.  Note that the event $\{\Me_n\approx \mu\}$ is a disjoint union of a large number of events $\{L^n=l\}$ for suitable $l\in\Ncal_n$. However, the probability of $\{\Me_n\approx \mu\}$ is dominated by the most likely sequence $l\in\Ncal_n$ (we will call this the \textit{recovery} sequence) in $\{\Me_n\approx \mu\}$. Indeed, we will show that the cardinality of such event does not play a role in the exponential rate of decay of its probability. Accordingly, the central part of the proof (both for the upper and for the lower bound) will revolve around finding the most likely sequence $l\in\Ncal_n$ that belongs to the set $\{\Me_n\approx \mu\}$.

To be precise, we fix $\mu\in \Mspace$ and we want to prove
\begin{align*}
-I(\mu)\leq\lim_{\delta\to 0}\liminf_{n\to\infty}\frac{1}{b_n^2}\log& \PP_{n,p_n(\theta)}\left(\Me_n(L^n)\in B_{\delta}(\mu)\right)\\
&\leq\lim_{\delta\to 0}\limsup_{n\to\infty}\frac{1}{b_n^2}\log \PP_{n,p_n(\theta)}\left(\Me_n(L^n)\in B_{\delta}(\mu)\right)=-I(\mu)
\end{align*}

where $B_{\delta}(\mu)$ is a ball of radius $\delta$ around $\mu$ in the vague topology. 
%Fix $\epsilon>0$ and let $\mu|_{\epsilon}\colon =\mathds{1}_{[\epsilon,\infty)}(x)\mu$ be the restriction to $[\epsilon,\infty)$ of a measure $\mu\in \Mspace$. We see that \[B_{\delta}(\mu)\subseteq \{l\in\Ncal_n\colon dm(\Me_n(l)|_{\epsilon},\mu|_{\epsilon})\leq \delta\} .\]

%Other problem: are there multiple typos in $F_{Ma}(l)$? I guess the product goes from $\beta_n$ to $n$ and there is a ${e}^{-kl_k}$ missing, right? Please, correct if it is wrong or explain if it is true!
%\marginnotemx{I don't think there is an $e^{kl_k}$ missing, it comes from the decomposition of $n!$ I used, at some point, I just made the majoration of $n!\leq n^k(n-k)!$\lu{Ok, and the previous question? I do not understand what happens with the $(1-p)$ terms.} I am checking this, it comes from the fact I want a probability for $F_{Mi}$}

\subsection{Main ingredients for the proof of Theorem~\ref{thmweakldp}}\label{sec:main_lem}
The main ingredients for the proof of Theorem~\ref{thmweakldp} are the following results, which we list here and we prove in separate sections. We first present a lemma to estimate and bound the size of a given subset of $\mathcal{N}_n$. Then in order to provide the large deviation principle% with rate $I(\mu)$
it will be sufficient to consider asymptotics of $\PP(L^n=l)$ for any $l$ such that $\Me_n(l)$ is close enough to $\mu$.

\begin{lemma}\label{lem:cardinality}
    For every $\mu\in\Mspace$ and every $\delta>0$, under Hypothesis~\ref{hypbn}, we have that
    \[
    |\{l\in\Ncal_n\colon \Me_n(l)\in B_{\delta}(\mu)\}|=e^{o(b_n^2)}.
    \]
\end{lemma}
The latter is proved in Section~\ref{appa}.

Then we introduce two  lemmas in which we estimate the contribution of respectively the macroscopic and the mesoscopic components to the probability mass function of $L^n$. Both lemmas are proved in Section~\ref{sec:lemmas_meso_macro}. First, let us focus on the contribution of macroscopic components.

\begin{lemma}\label{lem:macro}
For any $l^n\in \N_n$ such that $l^n_k\geq 1$ for some $k\geq \beta_n$, we obtain the following upper estimate
\begin{align*}
    \log\left(\PP_{n,p_n(\theta)}(L^n=l^n) \right)\leq -\frac 1 {8} \beta_n.
\end{align*}
\end{lemma}
From the above lemma we see that any configuration which has at least one connected component of macroscopic size occurs with a probability which is negligible in our regime.
In the following lemma we see that the contribution to $\PP_{n,p_n(\theta)}(L^n=l^n)$ given by the \textit{mesoscopic} terms runs on the scale $e^{O(b_n^2)}$.
%This lemma is completed with the contribution of the mesoscopic components given by lemma \ref{lem:meso}. 
%Both theorems are proved are proved in appendix \ref{appendix}. 

\begin{lemma}[Contribution of mesoscopic components]\label{lem:meso}
    Let $\mu\in\B$, where  $\mathcal B=\{\mu \in \M_\N(0,\infty), \int xd\mu(x)<\infty\}$. Let $0<\epsilon< \theta$, and $(l^n)_n$ be a sequence such that $l^n\in\Ncal_n$ with $l_k^n=0$ for all $k\geq \beta_n$.
    The following hold:
    \begin{itemize}
        \item [i)] whenever $\int_{\epsilon}^\infty x d \Me_n(l^n)(x)\geq 3\theta$, then
        \begin{align*}
    \frac 1 {b_n^2}\log(F_{Me}(l))\leq -h\left(\int_{\epsilon}^\infty x d \Me_n(l^n)(x)\right)^3 
    \end{align*}
    for some $h$ (positive and finite);
    \item [ii)] if $\int_\epsilon^\infty xd \Me_n(l^n)(x)$ is uniformly bounded, $\Me_n(l^n)$ converges vaguely to $\nu$ as $n\to\infty$ and for some $C>0$  $\int_{\epsilon}^C x d \Me_n(l^n)(x)\to \int_\epsilon^Cxd\mu(x)$, then 
    \begin{equation*}
\lim_{n\rightarrow \infty}\frac{1}{b_n^2}\log(F_{Me}(l))=-\frac 1 6 \left(\int_\epsilon^Cxd\mu(x)-\theta \right)^3-\frac{\theta^3}{6}+\frac{1}{24}\int_\epsilon^Cx^3d\mu(x).
\end{equation*}
where $F_{Me}$ is defined in \eqref{eq:F_me}.
    
    \end{itemize} 
    \end{lemma}
In order to prove the lemma above, one needs to have a good control on $z^n_k(\cdot)$, and thus on $\PP_{k,p}(L^k_k=1)$, for $k\propto (nb_n)^{2/3}$.  Precise asymptotics for the probability of an ER $\mathcal G(k,p)$ graph to be connected have been known for specific parameter regimes for quite some time; see \cite{LA21} for a summary. However, these are not enough in our near-critical regime and we prove the following proposition.
%we need to compute the leading behavior of \[\Pp_{K,p_n(\theta)}:=\PP_{K,p_n(\theta)}(L^K_K=1)=\PP_{K,p_n(\theta)}(\Gcal(K,p_n(\theta))\text{ is connected}),\] where $p_n(\theta)$ is defined in \eqref{eq:def_p} and $K=K_n\to\infty$ as $n\to\infty$.

\begin{proposition}[Probability of connectedness]\label{lemgraphconnex}
Let $K=K_n\to\infty$ be a diverging sequence, let $p_n(\theta)$ as in \eqref{eq:def_p} and 
\begin{align*}%
\Pp_{K,n}:=\PP_{K,p_n(\theta)}(L^K_K=1)=\PP_{K,p_n(\theta)}(\Gcal(K,p_n(\theta))\text{ is connected}).   
\end{align*}%
%
%Let $K_n \in \N$ be a divergent sequence, then
%and let $(b_n)_{n \in \N}$ be a sequence going to $\infty$ satisfying $b_n^2=o(n)$. Then,
%
Then
\begin{align*}
\log(\Pp_{K,n})=\log\left(p_{n}(\theta)^{K_n-1}K_n^{K_n-2}(1-p_n(\theta))^{\frac 1 2 (K_n-1)(K_n-2)}\right)+ C_n
\end{align*}
where $(C_n)_n$ is a sequence that depends on the growth rate of $K_n$. More specifically:
\begin{itemize}
    \item[(i)] If $K\in \N$ is such that $n^{2/5} \preceq K$ , and  $K=o(n^{3/4}b_n^{1/2})$, then

    \begin{align}\label{eq:prob_conn_regime_meso}
    C_n= \frac 1 6 K\left(\frac{K}{2n}\right)^2+O\left(K\left(\frac{K}{2n}\right)^2\left(\frac{b_n^2}{n}\right)^{1/3}+K\left(\frac{K}{2n}\right)^3+  \log(K)\right)
    \end{align}%
    In particular, if $K\sim u(nb_n)^{2/3}$ for some $u\in \mathbb R_+^*$, then
    \begin{align}
    \log(P_{K,n})=&K\log\left(\frac{n}{K}\right)+\theta K\left(\frac{b_n}{n}\right)^{1/3}-\frac{K^2}{2n}(1+\theta\left(\frac{b_n^2}{n}\right)^{1/3})\nonumber\\
    &+\frac{K^3}{24n^2}-\frac{\theta^2}2 K\left(\frac{b_n}{n}\right)^{2/3}+O(\log(K)).\label{eqpropconnexdeuxtrois}
    \end{align}
    \item[(ii)]
    If $K\preceq n^{2/5}$, then there is $C>0$ such that $C_n\leq CK\left(\frac{K}{n}\right)^{1/2}$. 
\end{itemize}
%
%In order to understand the form of $P_{u,n}$, we decompose the probability as follows
%\[
%P_{u,n}=\sum_{k=0}^{{{M}\choose 2}-M+1}C(M,M-1+k)p^{M+k-1}(1-p)^{{M\choose 2}-M-k+1},
%\]
%where $M=u (nb_n)^{2/3}$ and $C(M,H)$ is the number of connected graphs with $M$ vertices and $H$ edges, for $M-1\leq H\leq {M\choose 2}$. We hope that this sum is highly concentrated around its maximal term, therefore we will start by understanding which $k$ maximizes the sum. To do this, we need to use formulas to express the combinatorial quantity
%$C(M,M-1+k)$.
\end{proposition}

Finally, we need to know precise asymptotics for the probability that a graph at criticality does not have any components of mesoscopic size or larger. In order to do so, we fix $\epsilon>0$ and recall that  $\alpha_n=\epsilon (nb_n)^{2/3}$ for any $n\in \N$, as defined below \eqref{eq:z_k}.
Let \[E_{\alpha_n}:=\{l\in \mathcal{N}_n, l_k=0, \forall k,k\geq \alpha_n\}.\]
%\lu{I've made clear in the notation that $E$ depends on $m_n$, eventhough the estimate does not.}
The next proposition gives an estimate for the probability $\PP(L^n\in E_{\alpha_n})$ at an exponential scale.

\begin{proposition}\label{proppetitecomp}
%For any decreasing sequence $(m_n)_{n\in \N}\in \mathbb{R}_+$ converging to $0$, such that $m_n(nb_n)^{2/3}\rightarrow \infty$,
Fix $\epsilon>0$. % and let $\alpha_n=\epsilon (nb_n)^{2/3}$.
%\lu{I am not sure this works for any decreasing sequence! Take $m_n=c/(nb_n)^{2/3}$, then this is not true.}
For any sequence $N:=N(n) \sim n$, for any $\epsilon>0$, given  $\alpha_n=\epsilon (nb_n)^{2/3}$ and a random graph with law $\mathcal{G}(N,\frac{\omega_n}N)$, where $\omega_n:=1+\theta\left(\frac{b_n^2}{n}\right)^{1/3}+O\left(\left(\frac{b_n^2}{n}\right)^{2/3}\right)$, 
we have the following

\begin{align}
\liminf_{n\rightarrow\infty}\frac 1{b_n^2}\log &  \PP_{N,\frac{\omega_n}N}(L^N\in E_{\alpha_n})\geq %-\frac 16 \theta^3+\frac 16 \theta^3\vee 0=
     -\frac 16 \theta^31_{(\theta\geq 0)},\label{eqproppetitecomplow}\\
    \limsup_{n\rightarrow\infty}\frac 1{b_n^2}\log & \PP_{N,\frac{\omega_n}N}(L^N\in E_{\alpha_n})\leq %-\frac 16 \theta^3+\frac 16 \theta^3\vee 0=
     -\frac 16 \theta^31_{(\theta\geq 0)}+\delta_{\epsilon},\label{eqproppetitecompup}
     \end{align}
     where $\delta_{\epsilon}\searrow 0$ when $\epsilon\to0$.
\end{proposition}

\section{The proof of Theorem~\ref{thmweakldp}}\label{sec:proof_main_theorem}
In this section we prove the main theorem, using the results from Section~\ref{sec:main_lem}. We split the proof of the large deviation principle, as usually done, into a lower and an upper bound. 

\subsection{Lower bound}
Fix $\mu\in \{\M_{\N}(0,\infty), \int x d\mu<\infty\}$. One way to prove the desired lower bound is  to find a sequence $l^{n,\epsilon}$, where each $l^{n,\epsilon}\in\Ncal_n$ and such that $\Me_n(l^{n,\epsilon})\in B_{\delta}(\mu)$, for all $n$.
%It is enough to consider $\mu:=\sum_{i\in \Z}\delta_{u_i} \in \M_\N(0,\infty)$ with finite first moment and to estimate $\PP(Me_n(L^n)\in B_\delta(\mu))$ from below, 
First, we choose to study the truncated 
measure \[\mu^n_{\delta,\epsilon}:=\sum_{i,C_\delta \geq u_i\geq \epsilon}\delta_{\lfloor u_i (nb_n)^{2/3}\rfloor\big /(nb_n)^{2/3}} \in \M_\N(0,\infty)\]
with $\epsilon>0$ and $C_\delta>0$ being two constant, respectively small and large enough so 
that $\mu_{\delta,\epsilon}^n \in B_\delta(\mu)$ for any $n$ large. Notice that the sequence admits a limit $\mu_{\delta,\epsilon}\colon =\lim_{n\rightarrow\infty}\mu_{\delta,\epsilon}^n$, which satisfies $\mu_{\delta,\epsilon} \in \M_\N(0,\infty)$ and $\lim_{n\rightarrow\infty}\int xd \mu_{\delta,\epsilon}^n(x)=\int xd \mu_{\delta,\epsilon}(x)$, since the measures have all compact support. Moreover,  by construction, there are several sequences $l^{(n,\epsilon,\delta)}\in \mathcal N_n$ such that $\Me_n(l^{(n,\epsilon,\delta)})|_{[\epsilon,C_{\delta}]}=\mu_{\delta,\epsilon}^n$, where with $\alpha|_{[a,b]}$ we mean the restriction of the measure $\alpha\in\Mspace$ to the support $[a,b]$. In particular, for all these sequences the entries $l^{(n,\epsilon,\delta)}_k$ for $\lfloor \epsilon (nb_n)^{2/3}\rfloor \leq k\leq \lfloor C_{\delta}(nb_n)^{2/3}\rfloor $ are fixed, while the other entries are free (subject to the only assumption that $\sum_{k=1}^{n}kl^{(n,\epsilon, \delta)}_k=n$). 
%$\sum_{k=\alpha_n}^{\beta_n}\frac{k}{(nb_n)^{2/3}}\tilde l^{(n,\epsilon,\delta)}_k=\int_\epsilon^\infty xd\mu^n_{\delta,\epsilon}(x)$. 
We restrict our study to the sequences where $l^{(n,\epsilon, \delta)}_k=0$ for all $k>C_{\delta}(nb_n)^{2/3}$. All the entries $l^{(n,\epsilon, \delta)}_k$ for $k<\lfloor \epsilon (nb_n)^{2/3}\rfloor$ are free, as long as 
\[\sum_{k=1}^{\lfloor \epsilon (nb_n)^{2/3}\rfloor -1}kl^{(n,\epsilon, \delta)}_k=n-\int xd \mu_{\delta,\epsilon}^n(x)(nb_n)^{2/3}.
\]

Let us call $N:=n-\int xd \mu_{\delta,\epsilon}^n(x)(nb_n)^{2/3}$, we see that, as a consequence of all of the above, we have that
\[
\PP_{n,p_n(\theta)}\left(\Me_n(L^n)\in B_{\delta}(\mu)\right)\geq e^{o(b_n^2)}F_{Me}(l^{(n,\epsilon,\delta)})\sum_{\substack{l\in\Ncal_N\\ l_k=0,\, \forall k\geq \lfloor \epsilon (nb_n)^{2/3}\rfloor}}F_{Mi}(l).
\]
The rest of the proof relies on Lemma~\ref{lem:meso} and the following claim:
\begin{itemize}
%    \item {\bf Claim 1}:
%    \[
%    \frac1{b_n^2}\log F_{Me}(l^{(n,\epsilon,\delta)})\geq -\frac{\theta^3}{6}+\frac{1}{24}\int x^3 d\mu^n_{\delta,\epsilon}(x).
 %   \]
    \item {\bf Claim }:
    \[
    \sum_{\substack{l\in\Ncal_N\\ l_k=0,\, \forall k\geq \lfloor \epsilon (nb_n)^{2/3}\rfloor}}F_{Mi}(l)=\PP_{N,\frac{\omega_n}N}(L^N\in E_{\alpha_n}),
    \]
    where $\omega_n=1+(\theta-\int xd \mu_{\delta,\epsilon}(x))\left(\frac{b^2_n}{n}\right)^{1/3}+o\left(\frac{b^2_n}{n}\right)^{1/3}$.
\end{itemize}

Given the \textbf{claim} above, Lemma~\ref{lem:meso}  and Proposition~\ref{proppetitecomp}, we see that
\begin{align*}
\liminf_{n\to\infty} \frac{1}{b_n^2}&\log(\PP(Me_n(L^n)\in B_\delta(\mu)))\geq \lim_{n\rightarrow\infty} \frac{1}{b_n^2}\log(\PP(L^n=l^\delta))\\
&\geq -\limsup_{n\to\infty} I(\mu_{\delta,\epsilon}^n)\\
&\geq -I(\mu_{\delta,\epsilon})
\end{align*}
%&\geq \lim_{n\rightarrow \infty}-\frac 1 6 \left(\frac{1}{(nb_n)^{2/3}}\sum_{k=\alpha_n}^{\beta_n}kl^{(n,\epsilon, \delta)}_k-\theta \right)^31_{(\theta\leq \sum_{k=\alpha_n}^{\beta_n} \frac{kl^{(n,\epsilon, \delta)}_k}{(n b_n)^{2/3}})}\\
%&\qquad \qquad -\frac{\theta^3}{6}+\frac{1}{24}\sum_{k=\alpha_n}^{\beta_n}\left(\frac{k}{(nb_n)^{2/3}}\right)^3l^{(n,\epsilon, \delta)}_k\\

The last line comes from the fact that we actually have $I(\mu_{\delta,\epsilon}^n)\rightarrow I(\mu_{\delta,\epsilon})$ in that specific case. To conclude notice that 
in any case, $\lim_{\epsilon \rightarrow 0}\lim_{\delta\rightarrow 0}I(\mu_{\delta,\epsilon})=I(\mu)$;  indeed 
%if  $\int xd\mu(x)=\infty$, since 
%$\lim_{\delta \rightarrow 0}\lim_{\epsilon \rightarrow 0}\mu_{\delta,\epsilon}=\mu$ and $I(\cdot)$ is lower semi-continuous (see Proposition~\ref{propsemicon}), 
%we have $\lim_{\delta\rightarrow 0}\lim_{\epsilon \rightarrow 0}I(\mu_{\delta,\epsilon})\geq I(\mu)=\infty$, 
 we have both $\int xd\mu_{\delta,\epsilon}(x)\underset{\epsilon \rightarrow 0,\delta\rightarrow 0}\rightarrow \int xd\mu(x)$ and $\int x^3d\mu_{\delta,\epsilon}(x)\underset{\epsilon \rightarrow 0,\delta\rightarrow 0}\rightarrow \int x^3d\mu(x)$ and since $I(\cdot)$ is continuous with respect to these two parameters, $\lim_{\delta\rightarrow 0}\lim_{\epsilon \rightarrow 0}I(\mu_{\delta,\epsilon})= I(\mu)$.

%The proof of \textbf{Claim 1} works as the proof of Lemma~\ref{lem:meso}.\lu{I do not like this sentence too much, maybe we can reformulate the claim 1 in terms of Lemma~\ref{lem:meso} directly.} \textbf{Claim 2} can be deduced from \eqref{eqprimord}.
The proof of the \textbf{claim} can be deduced from the expression of $F_{Mi}(l)$ in \eqref{eq:F_mi} and \eqref{eqprimord} by noticing that for any $l\in E_{\alpha_n}$, $F_{Mi}(l)=\PP_{N,\frac{\omega_n}N}(L^N=l)$. 
% \lu{To prove the 2 claims: claim 2 is basically immediate, for claim 1 we apply the asymptotics for $\Pp_{K,p}$, the computation is the same for the upper bound, we can probably write a separate lemma that we can always cite.}

\subsection{Upper bound}

Fix $\mu\in\Mspace$ and $\delta>0$. 
Since we know from Lemma~\ref{lem:cardinality} that the cardinality of set 
\[
\{l\colon \Me_n(l)\in B_{\delta}(\mu)\}\subseteq \Ncal_n.
\]
is $e^{o(b_n^2)}$, it is enough to find a sequence $l^n$, such that $\Me_n(l^n)\in B_{\delta}(\mu)$ that maximizes the product $F_{Mi}(l^n)F_{Me}(l^n)F_{Ma}(l^n)$ in \eqref{eqprimord2}.

First of all, from Lemma~\ref{lem:macro}, we see that no sequence $l^n$ where $l^n_k\geq 1$ for $k\geq \beta_n$ is suitable to maximize this product. This means that an optimal $l^n$ will be such that $F_{Ma}(l^n)=1$. Now, we would like to 
 apply directly Lemma~\ref{lem:meso}-\textit{ii)} and Proposition~\ref{proppetitecomp} to approximate properly $F_{Mi}(l^n)F_{Me}(l^n)$ and then optimize it via standard optimization arguments. However, in order to apply those results, 
 we actually need the sum $\sum_{k\geq \alpha_n}\frac{k}{(nb_n)^{2/3}}l_k$ to converge, which is not 
the case in general for any sequence $l$ of elements of $\mathcal N_n$, even if $\Me_n(l)$ is close to a $\mu$ with finite first moment. \\
Therefore we need to control the upper bound also for elements $l\in \mathcal{N}_n$  such that $\sum_{k\geq \alpha_n}\frac{k}{(nb_n)^{2/3}}l_k$ does not converge. Let us focus on those $l$ first. Remember that, for the argument above, it is sufficient to focus on those $l$ such that $l_k=0$ for all $k\geq \beta_n$.
Let $(M_n)_{n\in \N}\in \mathbb{R}_+^\N$ be an increasing sequence such that $\lim_{n\rightarrow \infty}M_n = \infty$ and let $l$ be a sequence of elements of $\mathcal{N}_n$ such that $\sum_{k=\alpha_n}^{n}kl_k\geq M_n(nb_n)^{2/3}$. In that case either there is at least one macroscopic component or we see from Lemma~\ref{lem:meso}-\textit{i)}  that %\marginnotemx{$F_{Me}(l)\sim e^{-CM_n^3b_n^2}$ }
  \[F_{Me}(l)\leq e^{-h M_n^3b_n^2}\] for some $h>0$. 
  %\lu{Can we really use Lemma~\ref{lem:meso} ? It is not true, in general, that the first moment of $\Me_n$ is uniformly bounded.} 
  Thus, applying lemma~\ref{lem:macro} and lemma~\ref{lem:meso}-\textit{i)} for any $n$ large enough, for all $l$ such that $\sum_{k=\alpha_n}^{n}kl_k\geq M_n(nb_n)^{2/3}$, we have
\begin{align}\label{eqgmesogrosterm}
   \limsup_{n\to \infty} \frac{1}{b_n^2}\log\left(\PP_{n,p_{n}(\theta)}(L^n=l)\right)
 &\leq - \lim_{n\to \infty}h M_n^3.
\end{align}
Sending $n$ to infinity, the upper bound in \eqref{eqgmesogrosterm} goes to $-\infty$ and we see that no such sequence is suitable to maximise the product.
%\lu{Wait: the upper bound depends on $l$! You have to formulate it more carefully.}
%\lu{You do not need to consider $F_{Mi}(l)$ here,}
%{\color{blue} Notice that the terms $F_{Me}$ and $F_{Mi}$ are constructed respectively in \eqref{eq:F_me} and \eqref{eq:F_mi} so that their product correspond to roughly to a probability,  $F_{Mi}F_{Me}=\PP_{\tilde N,p}(L^{\tilde N}=l_{Me})e^{o(b_n^2)}$ with $\tilde N:=n-\sum_{k=\beta_n}^nkl_k$ and $l_{Me}$ is the element made from $\beta_n$ first indexes of $l$.}

%where the $+\theta^3$ in the last line comes from the fact that according to the polynomial equation \eqref{eqmesoprincip}, $\frac{1}{b_n^2}\log\left(F_{Me}(l)\right)\leq \theta^3$.
In particular, for any measure $\mu$ such that $\int_0^\infty x d\mu(x)=\infty$, the vague topology implies that there is a family $(\tilde M_\delta)_{\delta >0}$ with $\lim_{\delta \rightarrow 0}\tilde M_\delta=\infty$ such that for all $\nu \in B_{\delta}(\mu)$, $\int xd\nu(x)\geq \tilde M_\delta$. Thus for such measure $\mu$ 
\begin{align*}
\limsup_{n\rightarrow\infty}\frac{1}{b_n^2}\log(\PP(Me_n(L^n)\in B_\delta(\mu))\leq -\frac{\tilde M_\delta^3}{12}.
\end{align*}
If we send $\delta \to 0$ we see that we get precisely $-\infty$, which corresponds to $-I(\mu)$ in this case.\\

Suppose now that $\mu$ satisfies $\int_0^\infty x d\mu(x)<\infty$ and let $l\in \Me_n^{-1}(B_\delta(\mu))$ be a maximizing sequence of $\PP_{n,p}(L_n\in \Me_n^{-1}(B_\delta(\mu)))$, we see from \eqref{eqgmesogrosterm} that in order to obtain a non trivial upper bound, the sum
 $\frac{1}{(nb_n)^{2/3}}\sum_{k=\alpha_n}^{n}kl_k$ must be bounded. 
 %(in particular for $n$ large enough, we have $l_k=0$ for any $k\geq \beta_n$). 
 Let $M>0$ be an upper bound for it.
To such sequence one can represent $Me_n(l)$ as 
$Me_n(l):=\sum_{k\in \N}l_k\delta_{k/(nb_n)^{2/3}}=\sum_{i\in \Z}\delta_{u_i^l}$ and by a slight abuse of notation let $l$ be a subsequence such that $Me_n(l)$ converges to some limit $\nu^{\delta}\in \overline{B_\delta}(\mu)$. Let $\nu_+^l:=\sum_{i,u_i^l\geq \epsilon}\delta_{u_i^l}\in \M_\N(0,\infty)$ be the measure associated to the non microscopic components. Notice that $\nu_+^l$ depends on $l$ and thus on $n$, moreover as defined, it satisfies $\int_0^\infty x d\nu_+^l(x)\leq M$ and thus $\nu_+^l$ belong to a compact set $A_M$ of the form described in Proposition~\ref{prop:compact_sets}. We can then further extract from $l$ a sub-sequence so that $\nu_+^l$ converges to some measure $\nu_\epsilon$ and the total mass of it converge too :
\begin{align*}
\lim_{n\rightarrow \infty}\frac{1}{(nb_n)^{2/3}}\sum_{k=\alpha_n}^{\beta_n}kl_k&=\lim_{n\rightarrow \infty}\int_0^\infty xd\nu_+^l(x)\\
&=\lim_{n\rightarrow \infty}\int_\epsilon^Mxd\nu_+^l(x) =\int_\epsilon^Mxd\nu_\epsilon(x).   
\end{align*}
%We now choose $(M_n)_{n\in \N}$ such that $\lim_{n\rightarrow \infty} M_n = \infty$ and $M_n(nb_n)^{2/3}=o(n)$. 
Now we are able to apply Proposition~\ref{proppetitecomp}  to the product of microscopic components in the term $F_{Mi}(\cdot)$.
Indeed, for any such sub-sequence $l$ the product in the term $F_{Mi}(\cdot)$ corresponds to the probability  $\PP_{N,p'_n}(L_N= l_-)$, where $ l_- \in E_{\alpha_n}$ consists in the first $\alpha_n$ coordinates of $l$ and the random variable $L_N$ corresponds to the sequence of components of the random Erd\H{o}s-R\'enyi graph $\mathcal G(N,p')$ with $p'=\frac{1+\theta'}{N}$, $N:=n-\sum_{k=\alpha_n}^{\beta_n}kl_k$ and $\theta'=\theta-\sum_{k=\alpha_n}^{\beta_n}\frac{k}{(nb_n)^{2/3}}l_k=\theta -\int_\epsilon^Mxd\nu(x)+o(1)$. Therefore, the graph satisfies the assumptions of Proposition~\ref{proppetitecomp} and we can deduce from that proposition the following upper bound with $\delta_\epsilon$ converging to $0$ when $\epsilon$ tends to $0$ :

\begin{align*}
    \log&\left(F_{Mi}(l)\right)\leq-\frac 16 \left(\theta-\sum_{k=\alpha_n}^n \frac{k}{(nb_n)^{2/3}}l_k\right)^31_{(\theta\geq \sum_{k=\alpha_n}^n \frac{k}{(n b_n)^{2/3}})}+\delta_\epsilon b_n^2.
\end{align*}
Notice that  $(\nu_\epsilon)_{\epsilon>0}$ weakly converges to $\nu^\delta \in \overline{B_{\delta}}(\mu)$. Using Lemma~\ref{lem:meso} to control $F_{Me}$, we thus get for our subsequence of $l$
%\marginnotemx{From that point, add $+\epsilon$ pour la marge d'erreur des composantes microscopiques?}
\begin{align*}
\limsup_{n\rightarrow \infty}\frac{1}{b_n^2}\log\left(\PP_{n,p}(L^n=l)\right)&\leq\limsup_{n\rightarrow \infty}-\frac 1 6 \left(\frac{1}{(nb_n)^{2/3}}\sum_{k=\alpha_n}^{\beta_n}kl_k-\theta \right)^31_{(\theta\leq \sum_{k=\alpha_n}^{\beta_n} \frac{k}{(n b_n)^{2/3}})}\\
&-\frac{\theta^3}{6}+\frac 1 n\sum_{k=\alpha_n}^{\beta_n}kl_k\theta^3+\frac{1}{24}\sum_{k=\alpha_n}^{\beta_n}\left(\frac{k}{(nb_n)^{2/3}}\right)^3l_k+\delta_\epsilon\\
&\leq -\liminf_{n\to\infty} I(\nu_+^l)+\delta_\epsilon \\
&\leq - I(\nu_\epsilon)+\delta_\epsilon \leq -\inf_{\nu \in \overline{B_\delta}(\mu)}I(\nu)+\delta_\epsilon.
\end{align*}
Since $l$ is a subsequence of a maximizing sequence, thanks to Lemma~\ref{lem:cardinality}, we get
\begin{align*}
\limsup_{n\to\infty} \frac{1}{b_n^2}\log(\PP_{n,p}(Me_n(L^n)\in B_\delta(\mu))&=\limsup_{n\to\infty} \frac{1}{b_n^2}\log(\PP_{n,p}(L^n=l))\\
&\leq -\inf_{\nu \in \overline{B_\delta}(\mu)}I(\nu)+\delta_\epsilon,
\end{align*}
for any $\delta>0$. Then by the lower semi-continuity of $I(\cdot)$,
\begin{align*}
\lim_{\epsilon \rightarrow 0\delta \rightarrow 0}\limsup \frac{1}{b_n^2}\log(\PP_{n,p}(Me_n(L^n)\in B_\delta(\mu))&\leq -I(\mu),
\end{align*}
which ends the proof.

\section{Probability of connectedness: proof of Proposition~\ref{lemgraphconnex}}\label{Sec:prob_conn}

We recall that for any $K\in \N$ and $0<p<1$, \[
%\PP_{K,p}(\Gcal(K,p)\text{ is connected})
\Pp_{K,p}\geq p^{K-1}K^{K-2}(1-p)^{\frac 1 2 (K-1)(K-2)},\]
indeed the left hand-side corresponds to the probability that the graph is a tree spanning over the $K$ vertices. By abuse of notation, in the following, we will use $K$ to indicate the diverging sequence $K_n$ and $p$ to indicate the sequence $p_n(\theta)$.
%\gm{Needs rephrasing + explanation.} \marginnotemx{Done}
We start with the case (ii) when $K \preceq n^{2/5}$.
We consider the following rough upper bound consisting into the sum of the probability of having a connected graph with $k$ excess of edges with $k\leq \frac{K(K-1)}{2}$; %\gm{Needs explanation}\marginnotemx{Done}
\begin{align*}
\log\left(\Pp_{K,p}\right)&\leq \log\left( p^{K-1}K^{K-2}(1-p)^{\frac 1 2 (K-1)(K-2)}\left(1+\sum_{j=1}^{\frac{K(K-1)}{2}}{\frac{K(K-1)}{2}\choose j}\left(\frac p {1-p}\right)^j\right)\right)\\
&\leq \log\left( p^{K-1}K^{K-2}(1-p)^{\frac 1 2 (K-1)(K-2)}\right) +O\left(\frac{K^2}{n}\right),
\end{align*}
where the last line is a consequence of the fact that the term in parenthesis is equal to $(1+\frac p{1-p})^{\frac{K(K-1)}2}$.
%\lu{ The term in parenthesis is\[\log \left(1+\sum_{j=1}^{\frac{K(K-1)}{2}}{\frac{K(K-1)}{2}\choose j}\left(\frac p {1-p}\right)^j\right)\leq O\left( \frac{K^2}{n}\right)? \] Ah, maybe now I finally get it. Do you mean $N \choose k$ when you write $C^k_N$?? I've never seen this notation, please confirm and, in case,  change it to the usual one.}\marginnotemx{ $C^j_k$ means "$k$ among $j$" it is the combinatorial symbol. \lu{I think you mean "$j$ among $k$", but it is much more common to have $k\choose j$ to indicate such value, at least in our field. } We can revert to the other notation. \lu{yes} As for the inequality it is roughly a geometric sum with general term $pK\sim\frac{K^2}{n}$ initiated at $j=1$, we factorize this by $\frac{K^2}{n}$ and apply the limited developpment of the logarithm.\lu{Please, write at least a small sentence that justifies such an upper bound!}}
This concludes the proof of (ii).
%Thus $Kc_n^2 =O\left( \frac{K^2}{n}\right)$ as $n\to\infty$, uniformly for $K\preceq n^{2/5}$. 
%\lu{$O\left( \frac{K^2}{n}\right)$ is just the order of the upper bound!}

Next, we prove (i) by obtaining a more precise approximation of $\Pp_{K,p}$ for $K\succeq n^{2/5}$. To do so, we use a better approximation of the number of connected graphs $\tilde C(K,k)$ with exactly $k+K$ edges ($k\geq 1$), which is given in \cite{bdmckay}. The probability of having a specific connected graph with $K+k$ edges under $\mathbb{P}_{K,p}$ is then $p^{K+k}(1-p)^{\frac{K(K-1)}{2}-k-K}$.

For all parameter regimes, $\Pp_{K,p}$ can be written explicitly as
\begin{align}%
\Pp_{K,p} &= p^{K-1}(1-p)^{\frac 1 2 (K-1)(K-2)}(K^{K-2}+\tilde C(K,0)\frac{p}{1-p})\nonumber\\
&\quad+\sum_{k=1}^{\frac{K(K-1)}{2}}\tilde C(K,k)p^{k+K}(1-p)^{\frac{K(K-1)}{2}-k-K}.\label{eqsumconex}
\end{align}%
To give a rough upper bound of $\tilde C(K,0)$, we estimate in how many ways an edge can be added to any tree of size $K$. Since there are $K^{K-2}$ trees of size $K$ and at most $K \choose 2$ ways to add an edge, we conclude that
\begin{align*}%
\tilde C(K,0)=K^{K-2}O({K\choose 2}).
\end{align*}%
For $k\geq 1$, we use the estimate from \cite{bdmckay}:
uniformly in $K$ and $k>0$, we have the following
\begin{align*}
\tilde C(K,k)={{\frac{K(K-1)} 2}\choose {k+K}}\left(\frac{2e^{-x}y^{1-x}}{\sqrt{1-y^2}}\right)^K e^{a(x)}\left(1+O\left(\frac{1}{k}\right)+O\left(\frac{k^{1/16}}{K^{9/50}}\right)\right),
\end{align*}
where for convenience we set $x:=\frac{K+k}{K}$, and
\begin{align*}%
x\mapsto a(x):=x(x+1)(1-y)+\log(1-x+xy)-\frac{1}{2}\log(1-x+xy^2)
\end{align*}%
is a map on $[1,\infty)$ where $y\in [0,1)$ is defined by the following implicit formula,
\begin{equation*}%
xy=\frac 1 2\log\left(\frac{1+y}{1-y}\right).
\end{equation*}
Note that $k= K (x - 1)$. With this in mind, we set $C(K,x):=\tilde C(K,k)$. In order to get the asymptotics for $\Pp_{K,p}$, we aim to find the leading term in the sum in \eqref{eqsumconex} and to prove that every other term is negligible. Therefore, we look at the asymptotic behavior of 
\[x_n:=\text{argmax}_{x\in [1,\infty)}\left  \{C(K,x)\left(\frac{p}{1-p}\right)^{xK}\right\}\] 
as $n$ goes to infinity. We claim that for $1\ll K\ll n$  
\begin{equation}\label{eq:max_x_n}
    x_n=1+c_n^2+o(c_n^2),
\end{equation}
where $c_n:=\frac{K}{2n}+o\left(\frac{K}{2n}\right)$. For clarity, we split the proof of \eqref{eq:max_x_n} in two steps.\\
%\lu{Please, check if a rephrasing it like the one above works.} \\\lu{Later we will probably use that $c_n=e^{a_n}$ with\begin{align*}a_n&:=\left( \log (K-1)-\log(n)-\log(2)+\left(\theta\left(\frac{b_n^2}n\right)^{1/3}-\theta^2\left(\frac{b_n^2}n\right)^{2/3}\right)\right),\end{align*}so I leave the equation here.}

\noindent \textbf{Step 1}: {\it proof that $\lim_{n\to\infty}x_n=1$.}\\

Let us define $D_n : [1,\infty)\mapsto \mathbb R$ as %\gm{Which one of the two are we looking at here?}\marginnotemx{We look at $D_n$ what was the second you refered at?}
\begin{align*}%
{D_n(x)}&:= \log \left(C(K,x)\left(\frac{p}{1-p}\right)^{xK}\right)\\
&=\log \left({{\frac{K(K-1)} 2\choose {xK}}}\left(\frac{2e^{-x}y^{1-x}}{\sqrt{1-y^2}}\right)^K e^{a(x)}\left(\frac{p}{1-p}\right)^{xK}\right).
\end{align*}%
First    we prove that, for large $n$, the supremum of ${D_n(x)}$ is achieved for $x_n$ close to $1$. 

%\lu{Below, in the next page, we prove the sentence above. Still to be written down properly.}\\\lu{\textbf{Important:} where do I see that $K$ needs to have a certain speed for the sentence above to be true? If $K\propto n$ the sentence above is definitely not true, where do I see this?}

Using Stirling's expansions of the factorial terms and the dependence on $n$ of $p_n(\theta)$ (see \eqref{eq:def_p}), we obtain the following:

\begin{align*}
e^{D_n(x)}=&\frac{1}{\sqrt{2\pi xK}}\frac{\left(\frac{K(K-1)}{2}\right)^{\frac{K(K-1)}{2}}}{\left(\frac{K(K-1)}{2}-xK\right)^{\frac{K(K-1)}{2}-xK}(xK)^{xK}}e^{O(1/K)}\left(\frac{p}{1-p}\right)^{xK}\\
&\left(\frac{2e^{-x}y^{1-x}}{\sqrt{1-y^2}}\right)^K e^{a(x)}\left(1+O\left(\frac{1}{k}\right)+O\left(\frac{k^{1/16}}{K^{9/50}}\right)\right)\\
=&\frac{1}{\sqrt{2\pi xK(1-\frac{2x}{(K-1)})}}e^{xK+xK \log (K-1)-xK\log(2x)-x^2+O(\frac{x}{K})}\\
&\left(\frac{2e^{-x}y^{1-x}}{\sqrt{1-y^2}}\right)^K e^{a(x)}\left(1+O\left(\frac{1}{k}\right)+O\left(\frac{k^{1/16}}{K^{9/50}}\right)\right)\\
&e^{-xK\log(n)+xK\left(\theta\left(\frac{b_n^2}n\right)^{1/3}-\frac{1}{2}\theta^2\left(\frac{b_n^2}n\right)^{2/3}+O\left(\frac{b_n^2}{n}\right)\right)}.\\
\end{align*}%\marginnotemx{I find {\color{red} $-x^2$} as written for the term in red... for the $O_x(\frac 1 K)$ it is actually $O(\frac{x}{K})$ If I am right. }
%\lu{I believe the term in red is actually \textcolor{red}{$-2x^2$}, please check.\\\marginnotemx{no I think we can directly put the notation $O(\frac{x}{K})$}
%I believe that the term in blue is actually \textcolor{blue}{\[xK\left(\frac 1n+2\theta\left(\frac{b_n^2}{n^4}\right)^{1/3}+\theta^2\left(\frac{b_n^4}{n^5}\right)^{1/3}+
%o\left(\frac{b_n^{4/3}}{n^{5/3}}\right)
%\right)\]
%} please check.}\marginnotemx{I checked it the terms you added are right but negligible compared to $\frac{b_n^2}{n}$.}
Notice that $x\leq \frac{K-1}{2}$ and $a(x)=O(xK)$, thus, from the equation above, we get
\begin{align}\label{eqdn}
D_n(x)=-xK\log(\frac n K)-xK\log(2x)+O(xK).
\end{align}
From this equation we deduce that for $n$ large enough and any $x\geq 1$,
\begin{align}\label{eqmajordn}
D_n(x)\leq -\frac 1 2 xK\log(\frac n K)- \frac 1 2 xK\log(2x).
\end{align}
Notice that since $K\ll n$, the leading term in expansion \eqref{eqdn} is $K\log(\frac n K)$ and thus $\lim_{n\rightarrow \infty} \frac{D_n(x)}{K\log(\frac n K)}=-x$. We deduce that there is some $M>4$ such that for any $n$ large enough
\begin{align}\label{eqdnpourm}%
    D_n(1)> -\frac 1 4 MK\log(\frac n K)- \frac 1 4 MK\log(2M),
\end{align}%
Notice from equation \eqref{eqmajordn} that for any $x\geq M$, $D_n(x) < D_n(1)$, thus $\text{argmax} (D_n) \subset [1,M]$. Moreover, the limit $\lim_{n\rightarrow \infty} \frac{D_n(x)}{K\log(\frac n K)}=-x$ admits only one maximum on $[1,M]$, which corresponds to $x=1$. Since the sequence $D_n$ converges uniformly on each compact interval to this limit, the Proposition \ref{propunicon} below  applies on the sequence $(D_n(\cdot))_{n\in \N}$ and we deduce that $\lim_{n\rightarrow \infty} x_n=1=\text{argmax}_{x\in[1,M]}(-x)$.
 %Notice from the equation above that $(x_n)$ is bounded. 
 %Then, since $\frac{\log(K)}{\log(n)}$ is bounded, then  for any accumulation point $x^*$ of the sequence $(x_n)_{n \in\N}$ with some associated sub-sequence $(\tilde \phi(n))_{n\in \N}$, 
 %one can extract another sub-sequence $( \phi(n))_{n\in \N}$ such that $\frac{\log(K_{\phi(n)})}{\log(\phi(n))}$ converges to some limit $M\leq 1$ \lu{It must be $M<1$, I guess.}. Thus, for any $x\in [1,\infty)$, $\lim_{n\rightarrow\infty}\frac{D_{\phi(n)}(x)}{K\log(n)}=-(1- M)x$%\gm{Where does $1/3$ come from?}\marginnotemx{I rephrased it, the limit also depends of the behaviour of $K\log(K)$} 
 %which admits a unique maximum on $[1,\infty)$ when $x=1$. thus using Proposition~\ref{propunicon}, $x_{\phi(n)}\rightarrow 1$ and by uniqueness of the accumulation point, $\lim x_n=1$. 
 %\lu{I must admit I did not understand what is happening above. Are you trying to prove that $x_n\to 1$? In order to do what? }\marginnotemx{I want the finest approximation of $x_n$ as possible, first I prove that $x_n\rightarrow 1$ then I can expand $x_n$ around that value and figure out at which speed it goes to $1$. I am actually interested in $D_n(x_n)$.\lu{Please, explain in words such things! It is impossible for a reader to guess what's happening without any information on the aim. One has to read it at least twice in this way.}}

\begin{proposition}\label{propunicon}
Let $\left(f_n\right)_{n\in \N}\in C^0(X,\mathbb{R})^{\N}$ be a family of continuous functions on the metric space $X$ with compact support. Suppose $\left(f_n\right)_{n\in \N}$ converges uniformly to some continuous function $f$ that admits a unique maximum $x\in X$. Then any sequence $(x_n)_{n\in \N}$ such that $x_n \in argmax (f_n)$ converges to $x$.
\end{proposition}
We could not find a reference for this proof, so we added a short proof of it in Section~\ref{appa}.\\
%\marginnotemx{I didn't find any referenced proof, I put one in section 6.}

\noindent \textbf{Step 2}: {\it finer asymptotics for $x_n$.}\\

From the previous step we can assume that there is $M>0$ such that $x_n \leq M$ for any $n \in \N$, and we recall that the element $y\in [0,1)$ is defined as in the article \cite{bdmckay} by the bijection  
%\gm{Avoid notation $x$ for something that is not a free variable.\lu{I would kee}}
\begin{equation}
    \label{eq:expansion_x}
x(y)=\sum_{k=0}^\infty \frac{y^{2k}}{2k+1}=\frac{1}{2y}\log\left(\frac{1+y}{1-y}\right),\end{equation}
with element $y=0$ whenever $x=1$. 

To get a finer approximation for $x_n$, we now want to study the difference between $D_n(x)$ and $D_n(1)$ for $x(y)\leq M$ which means in particular that $y$ is bounded away from $1$: there is a constant $1>c>0$ such that $y<c$. 
We introduce the quantity $a_n$ so that \\
$D_n(1)=Ka_n+K\log(2)+O(\log(K))$. More concretely,
\begin{align}\label{eq:D_n_1_first_order}
a_n&:=\left( \log (K-1)-\log(n)-\log(2)+\left(\theta\left(\frac{b_n^2}n\right)^{1/3}-\theta^2\left(\frac{b_n^2}n\right)^{2/3}\right)+O\left(\frac{b_n^2}{n}\right)\right).
\end{align}
We also set $c_n:=\exp(a_n)=\frac{k}{2n}+o(1)$ and we define $\hat D_n :[0,c/c_n]\mapsto \mathbb R$ as an approximation of the difference between $D_n(x(c_n y))$ and $D_n(1)$ up to $\log(K)$, that is
\begin{align*}
D_n(x(c_ny))=D_n(1)+K\hat{D}_n(y)+O(\log(K)).
\end{align*}%
Using Taylor expansion, notice that the map
\begin{align*}%
x\mapsto a(x)=x(x+1)(1-y)+\log(1-x+xy)-\frac{1}{2}\log(1-x+xy^2)
\end{align*}%
is a bounded function in a neighborhood of $x=1$, thus an explicit expression of $\hat D_n(y)$ is
\begin{align*}
\hat D_n(y)=&a_n(x(c_ny)-1)-x(c_ny)\log(x(c_ny))\\&+(1-x(c_ny))\log(yc_n)
-\frac{1}{2}\log(1-y^2c_n^2)-\frac{(x(c_ny)^2-1)}{K}.
\end{align*}
Using \eqref{eq:expansion_x} 
%\lu{and assuming $c^2_nK\to \infty$ (do we need this here? I am not sure)}, we obtain
%
\begin{align}
\hat D_n(y)=&a_n\sum_{k=1}^\infty \frac{(c_ny)^{2k}}{2k+1}\nonumber\\
&-\sum_{k=0}^\infty \frac{(c_ny)^{2k}}{2k+1}\log(\sum_{k=0}^\infty \frac{(c_ny)^{2k}}{2k+1})-\sum_{k=1}^\infty \frac{(c_ny)^{2k}}{2k+1}\log(yc_n)\nonumber\\
&-\frac{1}{2}\log(1-y^2c_n^2)+O\left(\frac 1 K \sum_{k=1}^\infty \frac{(c_ny)^{2k}}{2k+1}\right)\nonumber\\
&=(a_n-\log(c_n))\frac{(c_ny)^{2}}{3}-\frac{(c_ny)^{2}}{3}-\frac{(c_ny)^{2}}{3}\log(y)\nonumber\\
&\qquad +\frac{1}{2}y^2c_n^2+O\left((yc_n)^4+\frac {(yc_n)^2} K\right)\label{eqaproxy}\\ 
\end{align}
%
%\label{eqaproxsuite}
%Notice that the $O_y(c_n^2)$ means that there is a function $\epsilon(.,.): (0,\infty)\times \N \mapsto (0,\infty)\times \N $ such that $\epsilon(y,.)$ is bounded for any $n\in \N$ and $\epsilon(.,n)$ is continuous (actually a power series in $y$). The last expansion \eqref{eqaproxsuite} makes sense when $a:=\lim_{n\rightarrow \infty}a_n-\log(c_n)$ exists in $(-\infty,\infty)$ and is different from zero.
 %Notice that if $a>0$ then $\frac{\hat D_n(y)}{|a|c_n^2}\simeq y^2$ for large $n$, thus its maximum is achieved when $y$ goes to $\infty$. On the other hand, when $a<0$ the maximum is obtained when $y$ goes to zero. Thus the right scale to look at the maximum would be when $a$ goes to zero. 
 %\lu{\textbf{Important}: I do not understand the sentences above. Do you mean that $a$ needs to be equal to $0$? Why, the reasoning is not super clear to me. Please, explain a bit better why the regime where $a:=\lim_{n\rightarrow \infty}a_n-\log(c_n)$ is finite is the sensible one.}
%
%Therefore, we take $c_n:=\exp(a_n)=\frac{K}{2n}+o\left(\frac{K}{2n}\left(\frac{b_n^2}{n}\right)^{1/3}\right)$, such that 
%which indeed satisfies assumption \ref{hypcn}. In particular,
%when $n^{2/3}=o(K)$,
we then have, from \eqref{eqaproxy}, 
\begin{align*}%
\frac{\hat D_n(y)}{c_n^2}=-\frac{y^{2}}{3}-\frac{y^{2}}{3}\log(y)+\frac{1}{2}y^2+O\left(y^2(yc_n)^2+\frac {y^2} {K}\right).
\end{align*}%
In particular, there is some $1<y_0$ such that for any $n$ large enough and $y\geq y_0$,
\begin{align*}
    \frac{\hat D_n(y)}{c_n^2}\leq -\frac{y^2}{12}\log(y).
\end{align*}
On $[0,y_0]$, $\frac{\hat D_n(y)}{c_n^2}$ converges uniformly to $D(y):=-\frac{y^{2}}{3}-\frac{y^{2}}{3}\log(y)+\frac{1}{2}y^2$ and
the unique maximizer of $D(y):=-\frac{y^{2}}{3}-\frac{y^{2}}{3}\log(y)+\frac{1}{2}y^2$ is $y=1$.
%, a quick computation of the derivative
%\footnote{$2y(\frac 1 6 -\frac 1 3 ln(y))$}
%gives a unique maximum at $y=1$. 
Thus applying Proposition~\ref{propunicon}, any maximizing sequence $(y_n)_{n\in \N}$ of $\hat{D}_n$ satisfies $y_n=1+o(1)$ and thus any maximizing sequence $x_n$ of $D_n$ satisfies 
\begin{align*}
x_n=1+c_n^2+o(c_n^2).
\end{align*}  
This concludes the proof of {Step 2} and thus the proof of \eqref{eq:max_x_n}.

As a consequence of \eqref{eq:max_x_n} we obtain that in a bounded neighborhood of $x=1$ (i.e., for any $c_ny$ in a bounded neighborhood of $0$),
\begin{align}%
D_n(x(c_ny))=Ka_n+K\hat{D}_n(y)+K\log(2)+O(\log(K)),\label{eqdecompDn}
\end{align}%
and therefore, for $y=1$,
\begin{align*}%
D_n(x(c_n))=Ka_n+\frac{K} 6 c_n^2+K\log(2)+O\left(Kc_n^4+Kc_n^2\left(\frac{b_n^2}{n}\right)^{1/3}+c_n^2+\log(K)\right).
\end{align*}%
%
%\marginnotemx{we want $Kc_n^4=o(b_n^2)$}
As a next step, we prove that the terms different from the argmax $x_n$ of
\begin{align*}
P_{K,p}&=\sum_{k=-1}^{\frac{K(K-1)}{2}}\tilde C(K,k)p^{k+K-1}(1-p)^{\frac{(K-2)(K-1)}{2}-k}
\end{align*}%
give a negligible contribution to $\Pp_{K,p}$. More precisely, we prove that for any $K\succeq n^{2/5}$,
\begin{align}
P_{K,p}&=(1-p)^{\frac{(K-2)(K-1)}{2}}e^{D_n(x(c_n))+O(\log(K))} \notag\\
&=p^{K-1}(1-p)^{\frac{(K-2)(K-1)}{2}}K^{K-2}e^{C_n}.\label{eqexpansionconnexe}
\end{align}
%
%Since 
%
%\begin{align*}
%C(K,x(yc_n))p^{xK}(1-p)^{\frac{K(K-1)}{2}-(xK)}=e^{Ka_n+K\hat{D}_n(y)+K\log(2)-\frac{(K-2)(K-1)} {2n} \left(1+\theta \left(\frac{b_n^2}{n}\right)\right)+o(\frac{K^2} {n^2})},
%\end{align*}
%
From equation \eqref{eqdn}, we know that for $n$ large enough and any $x\geq 1$,
$$
D_n(x)\leq -\frac 1 2 xK\log(\frac n K)- \frac 1 2 xK\log(2x)
$$
Thus if we fix some constant $M>0$ satisfying equation \eqref{eqdnpourm}, then
\begin{align*}
    \sum_{k=k_M}^{\frac{K(K-1)}{2}}e^{D_n(\frac{k+K}{K})} = O(1),
\end{align*}%
where $k_M:=\inf\{k\geq1: (k+K)/K\geq M\}$. We are left to investigate
$$
\sum_{k=1}^{MK}e^{D_n(\frac{k+K}{K})}.
$$  
According to equation \eqref{eqdecompDn}, the only term depending on $x$ in $D_n(x)$ is the term $\hat{D}_n(y)$. 
Thus it is enough to prove that 
\begin{align*}%
\sum_{k=1}^{MK}&C(K,\frac{k+K}K)p^{k+K}(1-p)^{\frac{K(K-1)}{2}-(k+K)}\\
&=e^{Ka_n+K\hat{D}_n(1)+K\log(2)-\frac{(K-2)(K-1)} {2n} \left(1+\theta \left(\frac{b_n^2}{n}\right)\right)+o(\frac{K^2} {n^2})}.
\end{align*}%
and that
\begin{align}\label{eqconcsum}
\sum_{k=1}^{MK}e^{K\hat{D}_n(y_{k,n})}=e^{K\hat{D}_n(1)+O(\log(K))},
\end{align}
where $y_{k,n}$ is the unique element $y\in \mathbb{R}_+$ such that $x(yc_n)=\frac{k+K}{K}$. Since $x(\cdot)$ is a non decreasing bijection from $(0,1)$ to $(1,\infty)$ for each $K\in \N$ we can associate to $k\in [K]$ a unique element $a_k\in (0,1)$ (depending on $K$) such that $\frac{K+k}{K}=x(a_k)$.
%We first deal with the case where $K\sim u(nb_n)^{2/3}$.
Let $m\in [K]$ be the smallest index such that $\hat D_n(y_{m,n}) \leq -2\log(K)$. The sum $\sum_{k\geq m}e^{K\hat D_n(y_{k,n})}$ is then bounded and we only need to investigate the size of $m$ to conclude. Let $k_0\in\N$ be the first index such that 
\begin{align*}
    a_{k_0}\geq
    \begin{cases}
        \left(\frac{2e\log(K)} K\right)^{1/2} &\text{ when }K\ll n^{2/3},     \\
        c_n^{1/2} &\text{otherwise}. 
    \end{cases}
\end{align*}
Then for $n$ large enough and any $k\geq k_0$, $y_{k,n} \geq \frac{a_k}{c_n}$ and
\begin{align*}
Kc_n^2\frac{y_{k,n}^2} 3 \left(\frac 1 2 -\log(y_{k,n})\right)&\ll-2\log(K).
\end{align*}
Thus from equation \eqref{eqaproxy}, for $n$ large enough, we have that $K\hat D_n(y_{k,n})\leq -2\log(K)$ for any $k\geq k_0$.
%\marginnotemx{$x=\frac{K+m}{K}$ and $xy=\log\left(\frac{1-y}{1+y}\right)$, a limited development around the value $y=0$ gives this expansion.}
Recall that $x=1+(yc_n)^2+O((yc_n)^4)$. Hence,
\begin{align*}%
k_0=Ka_{k_0}^2+O(Ka_{k_0}^4)=o(K),
\end{align*}%
and we deduce that $m\leq k_0=o(K)$ and obtain  \eqref{eqconcsum}.
Indeed, from the previous equation we deduce that $\sum_{k\geq m}e^{k\hat D_n(y_{k,n})}$ is bounded and thus 
\begin{align*}
   e^{\hat D_n(1)+O(\log(K))}\leq  \sum_{k\geq 1}^{KM}e^{k\hat D_n(y_{k,n})}&\leq\sum_{k\geq m}^{KM}e^{k\hat D_n(y_{k,n})}+me^{\hat D_n(1)+O(\log(K))}.
\end{align*}
It follows that $\sum_{k\geq 1}^{KM}e^{k\hat D_n(y_{k,n})}=e^{\hat D_n(1)+O(\log(K))}$, which is \eqref{eqconcsum}.

To summarize, \eqref{eqdecompDn} and \eqref{eqconcsum} together prove that 
\begin{align*}
    \Pp_{K,p}&=\sum_{k=1}^{MK} e^{D_n(x_k)}+O(1)=e^{Ka_n+K\log(2)+o(1)}\sum_{k=1}^{MK} e^{\hat D_n(y_{k,n})}\\
    &=e^{Ka_n+K\log(2)+\hat D_n(1)+O(\log(K))},
\end{align*}
where $a_n$ is defined in \eqref{eq:D_n_1_first_order}. The expansion of this latter term then gives the expression \eqref{eq:prob_conn_regime_meso}.
More precisely, for any $n^{2/5}\leq K$ such that $K=o(n^{3/4}b_n^{1/2})$,
\begin{align*}
\log(\Pp_{K,p})&=\log(p^{K-1}(1-p)^{\frac{(K-2)(K-1)}{2}}K^{K-2})+C_n.
\end{align*}
To conclude the point (i) of the Proposition~\ref{lemgraphconnex}, notice that in the particular case where $K\sim u(nb_n)^{2/3}$  then  $Kc_n^4\rightarrow 0$ and the error term is then  $O(\log(K))$ : 
%Thus behave  
%
%We now conclude for the point $(i)$ by replacing $K$ in \eqref{eqpkp} by it expression $K\sim u(nb_n)^{2/3}$ when $\log(n)=o(b_n^2)$ we can thoroughly replace $K$ in the above expression and obtain : 
%\begin{align*}
%P_{K,p}&=\sum_{k=1}^{\frac{K(K{\color{blue}-1})}{2}}\tilde C(K,k)p^{k+K}(1-p)^{\frac{K(K-1)}{2}-(k+K)}=e^{D_n(c_n)+o(b_n^2)}(1-p)^{\frac{K(K-1)}{2}}\\
%&=e^{D_n(x(c_n))+K\log(1-\frac{1+\theta\left(\frac{b_n^2}{n}\right)^{1/3}}{n})}.
%\end{align*}
%Thus,
%%\begin{align*}
%%D_n(x(c_ny))+K\log(1-\frac{1+\theta\left(\frac{b_n^2}{n}\right)^{1/3}}{n})&=Ka_n+K\hat{D}_n(y)+K\log(2)-\frac{K^2} {2n} \left(1+\theta\left(\frac{b_n^2}{n}\right)^{1/3}\right)+o(b_n^2)\\
%%&=Ka_n+K\hat{D}_n(y)+K\log(2)-\frac{u^2}{2}(nb_n^4)^{1/3}-\frac{\theta u^2}{2}b_n^2 +o(b_n^2)\\
%%\end{align*}
%
\begin{align*}
\log(\Pp_{K,n})&=K\log\left(\frac{n}{K}\right)+\theta K\left(\frac{b_n}{n}\right)^{1/3}-\frac{K^2}{2n}\left(1+\theta\left(\frac{b_n^2}{n}\right)^{1/3}\right)\\
&\quad+\frac{K^3}{24n^2}-\theta^2 K\left(\frac{b_n}{n}\right)^{2/3}+O(\log(K)).
\end{align*}
%

%\lu{\textbf{Important}: I am a bit worried about the error terms that we get from these extimates. When we do the upper and lower bounds in the LDP, some of these error terms will be multiplied by the multiplicity of components of that size and then summed all together. Has this been taken care properly? Do we use another upper or lower bound in that specific case (maybe it is enough). }\marginnotemx{I cared about it in the upper bound estimates.}
%The following paper is divided into two parts, dedicated to the study of combinatorial techniques to study the behavior of the probability of $L^n$ having either no large/medium-sized connected components ( section \ref{sectionpetitecomponent}) and either a component of size roughly $u(nb_n)^{2/3}$ whenever $u \in \mathbb{R_+}$ ( section \ref{sectionmediumcoomp}). 
    
\section{Probability of not having large components: proof of Proposition~\ref{proppetitecomp}}\label{sectionpetitecomponent}

We prove Proposition~\ref{eqproppetitecomplow} by separately showing the upper bound and the lower bound in the two following subsections. 
%the first one dealing with the lower bound  ($\leq$) and the second dealing with the upper bound ($\geq$) using Lemma~\ref{lemexpansionconnexe}. 
Notice that the right-hand side of the limits in Proposition~\ref{eqproppetitecomplow} depends heavily on the sign of $\theta$. Thus in each part of the proof the case $\theta >0$ and $\theta <0$ will be treated separately.

A crucial role is played by a special series defined as the sum of  $(\lambda_k(\omega))_{k\in\N}$, where, for $\omega\in[0,1]$ and $k\in\N$,
\[
\lambda_k(\omega):=k^{k-2}e^{-\omega k}\frac{\omega^k}{k!}.
\]
As we will see, this series is related to typical configuration of the graph and to the \textit{Borel distribution}. By the property of this {distribution} it is known that, for $\omega\in[0,1]$, we have
%\marginnotemx{Each time we cite this equation for the upper bound for the microscopic components\lu{what do you mean?} at each reference made to \ref{eqseries} in 5.2 we use this relation. we also use it to compute $\sum_{k=0}^\infty l_k$ and the approximations related to this.}
\begin{align}
   \sum_{k=1}^{\infty}k\lambda_k(\omega)= &\sum_{k=1}^{\infty} k^{k-1}e^{-\omega k}\frac{\omega^k}{k!}=\omega, \label{eqseries}\\
\sum_{k=1}^{\infty}\lambda_k(\omega)=&\sum_{k=1}^{\infty} k^{k-2}e^{-\omega k}\frac{\omega^k}{k!}=\omega(1-\frac {\omega}2)\label{eqseries2}.
\end{align}
We are going to exploit such properties to build a suitable recovery sequence, i.e., a sequence
$l^n(\omega_n)$ such that the probability of the event $\{L^n=l^n(\omega_n)\}$ will give us the desired lower bound to the probability of the event $E_{\alpha_n}$. %\lu{explain why this is important. {\bf Define what the term ``recovery sequence'' means!}}

In what follows, we fix $\omega$ as the sequence of elements $\omega_n:=1+\theta \left(\frac{b_n^2}{n}\right)^{1/3}+o(\left(\frac{b_n^2}{n}\right)^{1/3})$ 
for some $\theta \in \mathbb R$. And with a slight abuse, we fix the following notations:
\begin{align*}
    \lambda_k&:= \lambda_k(1)=k^{k-2}e^{-k}\frac{1}{k!}\\
    A_k &:=\frac{N}{\omega}\lambda_k e^{1+ka}
\end{align*}
with $a:=\log(\omega)-\frac{\omega}{2}$. To the purpose of the proofs, we also define the map $J_n(\cdot)$ which associates to $l\in \mathbb R^{\N}$ the following quantity
\begin{align}\label{eq:J_n}
    J_n(l):=\sum_{k\geq 1}l_k\log(\frac{A_k}{l_k}).
\end{align}
%
%\gm{Candidate change of notation: $J_n(\cdot)$ to $J_n(\cdot)$}
The proofs of the lower and upper bounds mainly consist in estimating this term $J_n(l)$ on the particular sequence $l^n$ and in proving that 
$$
\PP_{N,p}(L^{N}=l)=J_n(l)e^{o(b_n^2)}.
$$
\subsection{Proof of Proposition~\ref{proppetitecomp}: lower bound }\label{subseclowboundmicro}

To prove the lower bound, we provide some recovery sequence $l^\delta \in E_{\alpha_n}$ such that
\begin{align}\label{eqlowerboundmicro1}
\liminf_{n\rightarrow\infty}\frac 1{b_n^2}\log &  \PP_{N,p}(L^N=l^\delta)\geq %-\frac 16 \theta^3+\frac 16 \theta^3\vee 0=
     -\frac 16 \theta^31_{(\theta\geq 0)}.
\end{align}
{\bf Step 1: }{\it Construction of a recovery sequence $l^\delta \in E_{\alpha_n}$}.\\
We have to split our choice of the recovery sequence according to the sign of $\theta$. First we define $l\in \mathbb{R}^\N$ component-wise as 
\begin{align*}
l_k :=
\begin{cases}
\frac{N}{\omega}\lambda_k& \text{ when } \theta > 0,\\
\frac{1}{\omega}N\lambda_ke^{-k\xi} &\text{ when } \theta \leq 0.
\end{cases}
\end{align*}
with $\xi:=-(1-\omega +\log(\omega))$. From \eqref{eqseries} we get that
\begin{align}\label{eq:sum_opt_l}
\sum_{k=1}^{\infty} k l_k = 
\begin{cases}
\frac{N}{\omega}& \text{when } \theta >0, \\
N & \text{when } \theta \leq 0.
\end{cases}
\end{align}
%
%\lu{Are you sure that calling this $\xi$ and using $\delta$ for a different quantity is a good choice? Don't we have a better way to express this?}
The sequence $l^\delta$ is then made of elements of $E_{\alpha_n}$ obtained from $l$ by slight perturbations around the values of $l_k$, such that $l^{\delta}$ meets the constraint of having all integer entries with zero entries above $k=\epsilon (nb_n)^{2/3}$ and having total mass $\sum_{k=1}^Nkl_k=N$. Let us explain these perturbations. 
%\lu{{\bf IMPORTANT!! PLEASE:} explain in words what you are doing in the following, it takes at least 3 reads to understand what is going on and even after the third time it is hard to convince ourselves that $l^\delta:=l+\delta$ is an element of $E_{\alpha_n}$!!!!!!} 
First, we define $q\in\mathbb N$  as
%\label{eqrecoveringseq1}
\begin{align*}
q :=
\begin{cases}
\left\lfloor\sum_{k=1}^\infty k(l_k-\lfloor l_k\rfloor)+N(1-\frac 1 \omega)\right\rfloor &\text{ when } \theta > 0,\\
\left\lfloor\sum_{k=1}^\infty k(l_k-\lfloor l_k\rfloor)\right\rfloor &\text{ when } \theta \leq 0.
\end{cases}
\end{align*}
Let us explain the role of this $q$: it is supposed to compensate for the missing mass that is lost 
%if, when trying to meet the constraint of the entries being integer, 
when substituting $l_k$ with $\lfloor l_k \rfloor$ for any $k$ in order for the entries to be integer. Since $k\mapsto l_k$ is non-increasing, for all $k\geq\inf\{h\colon l_h<1\}$, we have $l_k-\lfloor l_k\rfloor\equiv l_k$. When $\theta>0$, $q$ also includes the missing mass due to \eqref{eq:sum_opt_l}, leading to the additional term $N(1-\frac 1 \omega)$.
We now choose some sequence $\delta$ of elements of $\mathbb R^\N$ such that $l^\delta=l+\delta$. Notice that, by abuse of notation, although $l$ and $l^{\delta}$ depend on $n$, we hide this dependence from the notation. Recall that $l^\delta \in E_{\alpha_n}$, which translates into the three following constraints on $l^\delta$.
\begin{enumerate}
\item $l^\delta_k\in \N$  for any $k\geq 0$. \\
\item $\sum_{k\leq \epsilon (nb_n)^{2/3}}kl_k^\delta=N$, \\
\item $l^\delta_k=0$ for any $k\geq \epsilon (nb_n)^{2/3}$.
\end{enumerate}
To satisfy the first and last conditions, we spread the missing mass $q$ over a set of indices above the (arbitrary) threshold $\lfloor n^{2/3}\rfloor$ and below the index $r$ defined as 
%\marginnotemx{Sorry I erased the explanations of how I chose the recovery sequence, but I think it is just to avoid putting the excess of mass in $q$ which may be a bit over $\epsilon (nb_n)^{2/3}$, there is no mistake.\lu{Ok, great. However I still do not understand what is $r$, is a quantity independent on $l$ now!! Are you sure???????? Please, correct if needed and delete my comments if they are addressed!} $q$ and $r$ "depend" on $l$ which depend on $\theta$ and $n$.}
%
\begin{align*}
r:=\sup\{j\geq \lfloor n^{2/3}\rfloor:\sum_{k=\lfloor n^{2/3}\rfloor}^jk\leq q\},
\end{align*}%
when this set is not empty. Effectively, this increases the value of $l_k$ by one for all $k\in \{\lfloor n^{2/3}\rfloor,\ldots,r\}$. If  $q<\lfloor n^{2/3}\rfloor$, this step is skipped and we will deal with the missing mass later.
The procedure above produces the sequence 
%\label{eqrecoveringseq1}
%
\begin{align*}
\tilde \delta_k:=
\begin{cases}
\lceil l_k\rceil-l_k &\text{ if } \lfloor n^{2/3}\rfloor\leq k\leq r,\\
\lfloor l_k\rfloor - l_k &\text{ otherwise},
\end{cases}
\end{align*}
where choosing $\lceil l_k\rceil-l_k $ when $k\in \{\lfloor n^{2/3}\rfloor,\ldots,r\}$ precisely does what we have just explained.
The sequence $l^{\tilde \delta}$ satisfies the first condition since the entries are integer.  
%\lu{I am sorry, but I do  not really see why all of the above would mean that $l^{\tilde \delta}$ satisfies the last condition. Where do I see that for $k\geq r$ for example $\lfloor l_k\rfloor=0$??? Is this a consequence of how we choose $r$?} 
We now change $\tilde \delta$ for $\delta$ so that $l^\delta$ satisfies $\sum_{k\leq \epsilon (nb_n)^{2/3}}kl_k^\delta=N$, such that it will satisfy also the second and third condition. 
%\lu{Sorry, but where are these sequences $(l_k^{\tilde \delta})_{k\in \N^*}$ and $(l_k^{\epsilon})_{k\in \N^*}$?? I see nowhere a definition. So far I see the definition of $l$, of $q$, of $r$ and of $\tilde \delta$. } 
Let $s:=q-\left\lfloor\sum_{k=\lfloor n^{2/3}\rfloor}^rk \right\rfloor$ when  $q\geq\lfloor n^{2/3}\rfloor$, and $s:=q$ otherwise. By definition of $r$, we see that $s$ is necessarily less or equal than $r$. Furthermore, let
%\lu{how can $r$ be $<\lfloor n^{2/3}\rfloor$ if it is defined like $r:=\sup\{j,\sum_{k=\lfloor n^{2/3}\rfloor}^jk\leq q\}$? }
%\label{eqrecoverseq2}
%
\begin{align*}
\delta_k:=
\begin{cases}
\tilde \delta_k+1 &\text{when } k=s,\\
\tilde \delta_k &\text{for other }k>1,
\end{cases}
\end{align*}
 and $\delta_1$ such that $l_1+\delta_1=N-\sum_{k=2}^\infty (l_k+\delta_k)$.
The sequence $l^\delta:=l+\delta$ is an element of $E_{\alpha_n}$ for $n$ large enough. Indeed, by construction we have $l^\delta \in \N^{\N}$ and $\sum kl_k^\delta=\sum kl_k+N(1-\frac 1 \omega)1_{\{\theta >0\}}=N$. To prove that $l^\delta_k=0$ for any $k\geq \epsilon (nb_n)^{2/3}$, we introduce the index $m$ defined as
\begin{align*}
    m:=\inf\{k\geq1:~l_k<1\}.
\end{align*}
Whatever the sign of $\theta$, we can give a rough estimate of $m$ from the definition of $l$. Indeed, we see that $l_k=O(Nk^{-5/2})$ and thus $m=O(N^{2/5})$. Since any element $l_k$ grows with $n$ we also know that $m\xrightarrow[n\to \infty]{}\infty$. In this way we see that, since $l_k$ is a decreasing sequence in $k$,  $l^\delta_k=0$ for any $k\geq m$ except whenever $k\in \{\lfloor n^{2/3}\rfloor ,\dots,r\}\cup\{s\}$. but the last index $r$ as defined is itself by construction of size 
\begin{align*}
r=O\left(\frac N {n^{2/3}}\right)=o(b_n^2)=o((nb_n)^{2/3}),
\end{align*}
thus $l^\delta$ has no positive coordinate above $\epsilon(nb_n)^{2/3}$ and so $l\in E_{\alpha_n}$.\\

\noindent {\bf Step 2: }{\it The sequence  $l^\delta$ satisfies \eqref{eqlowerboundmicro1}.}\\
%We now prove that $l^\delta$ satisfies the estimate \eqref{eqlowerboundmicro1}.
To compute $\PP_{N,p}(L^N=l^\delta)$, we use equation \eqref{eqprimord} with a lower estimate of the term $\Pp_{k,p}$ by the probability of having precisely the graph connected via a spanning tree:
$\Pp_{k,p}\geq k^{k-1}p^{k-1}(1-p)^{\frac12(k-2)(k-1)}$. 
Then we use  Stirling formula for all the factorials appearing in the lower bound. Recall the formula: for any $n\in \N^*$,  $n!=\sqrt {2\pi n} \left(\frac{n}{e}\right)^ne^{\frac{1+\gamma_n}{12n+1}}$ with $0\leq \gamma_n \leq 1$. Thus applying all of the above  to equation \eqref{eqprimord} for $l^\delta$ we get
\begin{align*}
\log&(\PP_{N,p}(L^N=l^\delta))\\ 
&\geq \sum_{k=1}^Nl^\delta_k\log\left( \frac{N^kp^{k-1}k^{k-2}(1-p)^{\frac{nk}{2}-3k-2}}{e^kk!l_k^\delta}\right) \\
&\quad+\frac{1}{2}\log(2\pi N)-\frac{1}{2}\sum_{k=1}^\infty\left(\frac{1+\gamma_{l^\delta_k}}{12l^\delta_k+1}+\log(2\pi l^\delta_k)\right)1_{\{l^\delta_k>1\}}\\
&\geq F(l^\delta)+o(b_n^2),
\end{align*}
%
%\label{eqmicropremier}
with the convention $0\log0=0$ and $F(l^\delta)$ defined as
\begin{align*}
F(l^\delta) &:= \frac{1}{2}\log(2\pi N)-\frac{1}{2}\sum_{k=1}^\infty\left(\frac{1+\gamma_{l^\delta_k}}{12l^\delta_k+1}+\log(2\pi l^\delta_k)\right)1_{\{l^\delta_k>1\}}\\
&\quad +\sum_{k=1}^\infty l^\delta_k\log\left(\frac{A_k}{l^\delta_k}\right)+O(1)\\
&=J_n(l^\delta)-\frac{1}{2}\sum_{k=1}^\infty\left(\frac{1+\gamma_{l^\delta_k}}{12l^\delta_k+1}+\log(2\pi l^\delta_k)\right)1_{\{l^\delta_k>1\}}+o(b_n^2),
\end{align*}
where we use the definition of  $J_n(l^\delta)$ in \eqref{eqlowerboundmicro1} and the fact that $\log N=o(b_n^2)$ to go from the first to the second line. Therefore we see that we can lower bound the logarithm of the probability with the function $J_n$ computed in $l^\delta$ plus an error that comes from the use of Stirling formula.

First we focus on computing the value of $J_n(l^\delta)$, which 
%Through algebraic manipulations we see that
%Recall that $m = O(n^{2/5})$ and $l^\delta_k=0$ for any $k\geq m$ (except whenever $k\in \{\lfloor n^{2/3}\rfloor ,\dots,r\}\cup\{s\}$), 
we rewrite as
%
%\label{eqapproxentier}
\begin{align}
J_n(l^\delta)&=J_n(l)+\sum_{k}\delta_k+a\sum_k k \delta_k-\sum_{k}(l_k+\delta_k)\log(1+\frac{\delta_k}{l_k}) \nonumber\\
&=\tilde J_n(l)+\sum_{k}\delta_k -\sum_{k}(l_k+\delta_k)\log(1+\frac{\delta_k}{l_k})\label{eqdevabovewind} 
\end{align}
where $\tilde J_n(l)$ is 
\begin{align*}
\tilde J_n(l) := J_n(l) +a\sum_{k=1}^nk\delta_k,
\end{align*}
and $a=\log( \omega )-\frac {\omega}2$ was defined at the beginning of this section. Notice that $\tilde J_n(l)$ does not depend on $\delta$, since, by construction, 
\begin{align*}
\sum_{k}k\delta_k :=
\begin{cases}
N(1-\frac1{\omega})& \text{ when } \theta > 0,\\
0 &\text{ when } \theta \leq 0.
\end{cases}
\end{align*}
Let us underline that $\tilde J_n(l)$ accounts for the leading terms of $J_n(l^\delta)$, we will prove later that the rest is negligible.

The sequence $l_k$ has been chosen so that whatever the sign of $\theta$ is, $\tilde J_n(l)$ corresponds to the estimate in \eqref{eqlowerboundmicro1}:
\begin{align*}
\tilde J_n(l)=
\begin{cases}
-\theta^3b_n^2+o(b_n^2) &\text{ when } \theta > 0,\\
o(b_n^2) &\text{ when } \theta \leq 0.
\end{cases}
\end{align*}
To see this, note that when $\theta \leq 0$, $\sum_{k=1}^nkl_k=N$ (in other words $\sum_{k=1}^nk\delta_k=0$) and by direct application of equations \eqref{eqseries} and \eqref{eqseries2},
\begin{align*}
\tilde J_n(l)&=N( a+\xi)+N(1-\frac{\omega}2) \\
&=\left(\omega-1-\frac{\omega}{2}+O(\frac 1 N)\right)+1-\frac{\omega}2=o(b_n^2).  
\end{align*}
On the other hand, when $\theta>0$,  we have $\sum_{k=1}^nkl_k= N / \omega$,  equivalently $\sum_{k=1}^nk\delta_k=N(1-\frac 1 \omega)$, thus applying again the equation in \eqref{eqseries2},
\begin{align*}
\tilde J_n(l)&=aN(1-\frac 1 \omega)+J_n(l)=aN(1-\frac 1 \omega)+ a\frac N \omega+N\frac{\omega}{2}\\
&=-\frac 1 6 \theta^3b_n^2+o(b_n^2).\nonumber  
\end{align*}
Now we investigate the remaining terms in \eqref{eqdevabovewind} and prove they are negligible. We substitute $\delta$ with its explicit expression in terms of $l$.
Recall that $l_k$ is a decreasing sequence and  $l^\delta_k=0$ for any $k\geq m$, except whenever $k\in \{n^{2/3},\dots,r\}\cup\{s\}$. 
%Next, we provide a lower bound for $J_n(l^\delta)$.
%
\begin{align}
J_n(l^\delta)-\tilde J_n(l)&\geq \sum_{k=1}^\infty \lfloor l_k\rfloor-l_k  \nonumber\\
&-\sum_{k>1}^m \lfloor l_k\rfloor \log(\frac{\lfloor l_k\rfloor}{l_k})- (\lceil l_1\rceil+2) \log(\frac{\lceil l_1\rceil+2}{l_1})\nonumber\\
&-2\sum_{k=\lfloor n^{2/3}\rfloor}^r \lceil l_k\rceil \log(\frac{\lceil l_k\rceil}{l_k})-\lceil l_s\rceil \log(\frac{\lceil l_s\rceil}{l_s})\nonumber\\
&\geq -m-\sum_{k\geq m}^\infty l_k\label{eqlowerabove}\\
&-2\sum_{k=\lfloor n^{2/3}\rfloor}^r \lceil l_k\rceil \log(\frac{\lceil l_k\rceil}{l_k})-\lceil l_s\rceil \log(\frac{\lceil l_s\rceil}{l_s})\label{eqlowerabove2}\\
&- (\lceil l_1\rceil+2) \log(\frac{\lceil l_1\rceil+2}{l_1})\label{eqlowerabove3}
\end{align}
We now show that each of he three different terms on the right-hand side is $o(b_n^2)$. Recall that $r=O(\frac N {n^{2/3}})=o(b_n^2)$ and thus the term in \eqref{eqlowerabove2} is $O(2r\log(2))=o(b_n^2)$. For the last term \eqref{eqlowerabove3}, notice that for any $x\in \mathbb R_+^*$, $\log(\frac{\lceil x\rceil+n}{x})\leq \frac {1+n} x$ and thus the last term \eqref{eqlowerabove3} is bounded from above by $8+\log(l_s)+\log(l_1)=O(\log(n))=o(b_n^2)$.

We are now left to bound the sum $\sum_{k\geq m}l_k$. Recall that $\frac 1 \omega \lambda_k N\underset{k\rightarrow \infty}\sim  \frac{N}{\sqrt{2\pi}k^{5/2}\omega}$. If $\theta > 0$, then $m\sim \frac{N^{2/5}}{\sqrt{2\pi}\omega}$, and  thus
\begin{align*}
\sum_{k\geq m}l_k &\sim \sum_{k\geq m}\frac{N}{\sqrt{2\pi}k^{5/2}\omega}\sim \frac{N}{\sqrt{2\pi}m^{3/2}\omega}\\
&\sim \frac{N^{2/5}}{\sqrt{2\pi}\omega}=o(b_n^2).
\end{align*} 
Suppose now that $\theta \leq 0$, then
$f(t):=\frac{N}{\sqrt{2\pi}t^{5/2}\omega}$, since $m \rightarrow \infty$, the comparison between series and integrals allows us to write
\begin{align}%
    \frac{N}{\omega}\sum_{k\geq m} \lambda_k e^{-k\xi_n}&\sim \int_{m}^{\infty}f(t)e^{-t\xi_n}\mathrm{d}t\\
    &\sim \frac{1}{\xi_n } \int_{m\xi_n}^{\infty}f\left(\frac{u}{\xi_n}\right)e^{-u}\mathrm{d}u\\
    &\sim \xi_n^{3/2}\int_{m\xi_n}^{\infty}\frac{N}{\sqrt{2\pi}u^{5/2}\omega}e^{-u}\mathrm{d}u.\label{eqsumconvmicro}
\end{align}%
We fix some $\epsilon$ so that $0<\epsilon<\frac 1 3$, then,  since $\lambda_k\sim \alpha \frac 1 {k^{5/2}}$ for some $\alpha>0$, whenever $m\geq N^{\frac 2 5 (1-\epsilon)}$ the sum $\sum_{k\geq m} l_k\leq \frac{N}{\omega}\sum_{k\geq m} \lambda_k e^{-k\xi_n}=o(n^{4/10})=o(b_n^2)$ according to hypothesis \ref{hypbn} made on $b_n$  .
Thus if we split the sequence $(\frac 1 {b_n^2}\sum_{k\geq m_n}l_k)_{n\geq0}$ into two subsequence:
\begin{itemize}
\item[(i)]   one made of   the indexes $n$ such that $m_n\xi_n\geq \epsilon \log(N)$,
\item[(ii)] the other made of all the other indexes.
\end{itemize}
  The sequence $(i)$ converges from the convergence to $0$ of the integral \eqref{eqsumconvmicro}. As for the second, we notice that by definition of $m$ and the fact that $\lambda_k\sim \alpha \frac 1 {k^{5/2}}$, 
\begin{align*}%
    \frac{N}{\omega}\lambda_me^{-m\xi_n}\leq 1,
\end{align*}%
thus,
\begin{align*}
m\geq \frac 1 {\omega \alpha}N^{\frac 2 5 (1-\epsilon)}.
\end{align*}
Then according to what we said about the value of $\epsilon$, that subsequence also converges to $0$. Thus $\sum_{k\geq m}l_k=o(b_n^2)$  in any case and the term \eqref{eqlowerabove} is also $o(b_n^2)$.
\begin{align*}
J_n(l^\delta)\geq \tilde J_n(l)+o(b_n^2).
\end{align*}%
Finally notice that since $m=O(N^{2/5})$, 
\begin{align*}
F(l^\delta)\geq J_n(l^\delta)-m\log(N)+o(b_n^2), 
\end{align*}
thus $F(l^\delta)=J_n(l^\delta)+o(b_n^2)$, proving the identity 
\begin{align*}
\lim_{n\rightarrow \infty}\frac{1}{b_n^2}\log(\PP_{N,p}(L_k=l_k^\delta))=-\frac{1}{6}\theta^31_{(\theta\geq 0)}.
\end{align*}%

\subsection{Proof of Proposition~\ref{proppetitecomp}: upper bound}

Since $\PP_{N,p}(L^N\in E_{\alpha_n})$ is a probability, the upper bound when $\theta < 0$ is trivial.
Thus we focus here on the proof of the upper bound when $\theta \geq 0$.

%For any sequence $l\in E_{\alpha_n}$,
%We keep the notation $(l_k)_{k\in \N}$ for the sequence defined by the equation \eqref{eqlk} any other sequence in $E_{\alpha_n}$ will be represented as $(l_k^\epsilon)_{k\in \N}$ with coefficient satisfying for any $k\in \N$, $l_k^\epsilon=l_k+\delta_k$ and $\sum_{k\in \N}k\delta_k=N(1-\frac 1 \omega)$. 
%\lu{Are you sure that this is enough? I am not sure that any sequence in $E_{\alpha_n}$ can be represented like this. Think of $l$ such that $l_1=N$ and $l_k=0$ $\forall k>1$, this cannot be represented as above. Am I wrong?}\marginnotemx{it was some remnants of previous redaction, I removed it.}

%\lu{please, explain better how you get the upper bound. Which part of Lemma~\ref{lemgraphconnex} you use and where. It is extremely hard for a reader to jump from one equation to the other inside the paper just to understand what is written.}

For any $l\in  E_{\alpha_n}$, since $\sup\{k, l_k>0\}=o(N^{3/4}b_n^{1/2})$, 
%(note that this is equivalent to $\beta_n$), we get the following 
in order to get an upper estimate of $\PP_{N,p}(L^N=l)$, we can use equation \eqref{eqprimord} and  the approximations  given by point (i) and (ii) of Proposition~\ref{lemgraphconnex}.
%in equation \eqref{eqprimord}.
%leads into adding the term $Kc_n^2$ into the equation \eqref{eqmicropremier} of page \pageref{eqmicropremier}  we get the following estimate,
This means that there is some $C>0$ such that for any $n\in \N$, 

\begin{align}
\log(\PP_{N,p}(L^N=l))\leq&
\sum_{k=1}^Nl_k\log\left( \frac{N^kp^{k-1}k^{k-2}(1-p)^{\frac{nk}{2}-3k-2}}{e^kk!l_k}\right)+\sum_{k=1}^Nl_kC_n(k)\nonumber\\
&+\frac{1}{2}\log(2\pi N)-\frac{1}{2}\sum_{k=1}^\infty\left(\frac{1+\gamma_{l_k}}{12l_k+1}+\log(2\pi l_k)\right)1_{\{l_k>1\}}\nonumber\\
\leq& 
J_n(l)+\frac 1 6\sum_{N\geq k\geq n^{2/5}}\frac{k^3}{4n^2}l_k+o(b_n^2)\label{eqmicroupprincip}\\
&+C\sum_{N\geq k\geq n^{2/5}}l_k\left(k\left(\frac{k}{2n}\right)^2\left(\frac{b_n^2}{n}\right)^{1/3}+k\left(\frac{k}{2n}\right)^3+\log(k)\right),\nonumber
\end{align} 

where $J_n$ is defined in \eqref{eq:J_n}. The second inequality comes from removing the negative quantity $\sum_{k=1}^\infty\left(\frac{1+\gamma_{l_k}}{12l_k+1}+\log(2\pi l_k)\right)1_{\{l_k>1\}}$.
{
First, notice that the sum in the last line of the inequality above is of size $\epsilon b_n^2+o(b_n^2)$. Indeed, since $\sum_{k=1}^nkl_k=N$ and $l_k=0$ for all $k\geq \epsilon (nb_n)^{2/3}$, 
\begin{align*}
   \sum_{n^{2/5}\leq k\leq \epsilon(nb_n)^{2/3}}&l_k\left(k\left(\frac{k}{2n}\right)^2\left(\frac{b_n^2}{n}\right)^{1/3}+k\left(\frac{k}{2n}\right)^3+  \log(k)\right)\\
   &\leq
   \epsilon b_n^2+C\log(N)\sum_{n^{2/5}\leq k\leq \epsilon(nb_n)^{2/3}}l_k+o(b_n^2)\\
   &\leq \epsilon b_n^2+ \frac{N}{n^{2/5}}\log(N)+o(b_n^2)\\
   &\leq \epsilon b_n^2+o(b_n^2).
\end{align*}
The leading terms of $\log \PP_{N,p}(L^n=(l_k)_{k\in \N})$ are thus those in line \eqref{eqmicroupprincip}. We will now find a suitable upper bound for % in particular  we have to control the size of 
$J_n(l)+2\sum_{N\geq k\geq n^{2/5}}\frac{k^3}{4n^2}l_k$ optimizing this sum over $E_{\alpha_n}$. In order to do so, we fix $M\in \mathbb R$ and denote by $H^{(n)}_M$ the following level set:
\begin{align*}
     H^{(n)}_M:=\{l\in \mathbb R_+^{\N}\colon& \sum_{ n^{2/5}\leq  k \leq \epsilon(nb_n)^{2/3}}kl_k=(1-\frac{1}{\omega})N+M(nb_n)^{2/3}\\
     &\qquad \textit{ and }\sum_{1\leq k \leq \epsilon(nb_n)^{2/3} }kl_k=N\}.
 \end{align*}
Now, let \[l_n^M\colon =\arg \max_{l\in H^{(n)}_M}\{J_n(l)+2\sum_{N\geq k\geq n^{2/5}}\frac{k^3}{4n^2}l_k\}.
\] 
%be the sequence of elements of $H^{(n)}_M$ maximizing the left hand side of equation \eqref{eqmicroupprincip}. 
As previously, by abuse of notation, we omit the dependence on $n$ of the sequence and we just use the notation $l^M$. Note that  we are optimizing over sets where the total mass above $n^{2/5}$ is fixed and equal to $(1-\frac{1}{\omega})N+M(nb_n)^{2/3}$. If $M$ is varying in the set 
\[K_n:=\left[-(1-\frac{1}{\omega})\left(\frac{n}{b_n^2}\right)^{1/3};\frac{1}{\omega}\frac{N}{(nb_n)^{2/3}}\right],\] 
it means that we actually cover the whole set $E_{\alpha_n}$ with the union of $ H^{(n)}_M$ and we optimize indeed over a set that contains $E_{\alpha_n}$. However, we will show later that the optimal $l$ belongs to one of these level sets where $M$ is uniformly (in $n$) bounded from above.  Note that, in that case, for a fixed $M\in\R$, $(1-\frac{1}{\omega})N+M(nb_n)^{2/3}$ is a quantity proportional to $(nb_n)^{2/3}$, i.e. the total mass above $n^{2/5}$ is indeed small with respect to $N$. 
 
For any $l\in E_{\alpha_n}$, we have the following estimate: 
%\marginnotemx{We optimize taking into account that $M$ can takes negative values too which would then mean that all the mass may be concentrated below $n^{2/5}$}
%\lu{WHY? I do not directly see how the term $\log(\PP_{N,p}(L^N=(l_k)_{k\in \N}))$ for any $l\in E_{\alpha_n}$ is maximised by finding the optimal way to put an amount of mass equal to $(1-\frac{1}{\omega})N+M(nb_n)^{2/3}$ into components of size $n^{2/5}\leq  k \leq \epsilon(nb_n)^{2/3}$ and then optimizing over $M$. Why is this better than having all the mass in components smaller than $n^{2/5}$? This is not explained anywhere and it sounds surprising to a reader! Moreover, I suggest to remember to readers that $\omega$ is indeed a quantity depending on $n$ and on $\theta$!  }
\begin{align}
   \log(\PP_{N,p}(L^N=l))&\leq \sup_{M\in K_n } \left(J_n(l^M)+2\sum_{n^{2/5}\leq k\leq \epsilon(nb_n)^{2/3}}\frac{k^3}{4n^2}l^M_k\right)+o(b^2_n)\nonumber\\
   &\leq \sup_{M\in K_n }\left(-\sum_{1\leq k\leq \epsilon(nb_n)^{2/3}} l^M_k\log(\frac{\omega l^M_k}{N\lambda_ke})+2\epsilon^2 (M+\theta)b_n^2\right)\notag\\
   &\quad+o(b^2_n)+aN.\label{eqmicroupsumj}
\end{align}
%
% with $K_n:=[-(1-\frac{1}{\omega})\left(\frac{n}{b_n^2}\right)^{1/3};\frac{1}{\omega}\frac{N}{(nb_n)^{2/3}}]$.
In order to estimate the sum involved in the last equation, we split it in two parts, one which is the sum over the indexes $k\leq n^{2/5}$ and the other regrouping the rest. Using a standard derivation method one can see that each term of index $k$ in the sum in \eqref{eqmicroupsumj} attains it maximum at some value $\tilde l_k:=\frac{N}{\omega}\lambda_k$ and from the definition of $\lambda_k$ we deduce that $\lambda_k=O(k^{-5/2})$ and thus, for any $M$,
\begin{align*}
   - \sum_{n^{2/5}< k\leq \epsilon(nb_n)^{2/3}} l^M_k\log(\frac{\omega l^M_k}{N\lambda_ke})\leq  \sum_{n^{2/5}\leq k\leq \epsilon(nb_n)^{2/3}}\frac{N}{\omega}\lambda_k=O(n^{2/5})=o(b_n^2).
\end{align*}
From now we focus on estimating the sum in \eqref{eqmicroupsumj} over the indexes $k\leq n^{2/5}$. Note that we can further upper bound that term in the following way:
\begin{align}\label{eqminorupfunction}
-\sum_{1\leq k\leq n^{2/5}} l^M_k\log(\frac{\omega l^M_k}{N\lambda_ke})\leq \sup_{l\in H^{(n)}_M} -\sum_{1\leq k\leq n^{2/5}} l_k\log(\frac{\omega l_k}{N\lambda_ke}).   
\end{align}
%To do so we give an estimate to the supremum on the right-hand side for large $n$, using the standard technique of the Lagrangian multipliers.
We are looking for a supremum over $H^{(n)}_M$, which is achieved for a sequence $\tilde{l}^M\in H^{(n)}_M$ and therefore it satisfies the constraint $\sum_{k\leq n^{2/5}}k\tilde{l}^M_k=\frac{N}{\omega}-M(nb_n)^{2/3}$. 
%Using the standard technique of the Lagrangian multipliers to express $\tilde{l}^M_k$, 
To find such a $\tilde{l}^M$, first, we notice that whenever $M>0$, for $n$ large enough, $\frac{N}{\omega}\sum_{k=0}^{n^{2/5}}k\lambda_k\geq \frac{N}{\omega}-M(nb_n)^{2/3}$. Then we use the standard technique of the Lagrangian multipliers to express $\tilde{l}^M_k$, i.e. there is a non-negative sequence $(\delta_n^M)_{n\in \N}$ such that:

%component wise of the left hand side of \eqref{equpmajotot} which contains the leading terms (i.e for a fixed $k\in \N$ we look at the element $l_k \in \mathbb{R}_+$ maximizing the corresponding term in the sums) : a standard derivation tells us that for $k\geq n^{2/5}$, the maximum is attained for $\tilde{l}\in \mathcal{N}_N$ satisfying 
\begin{align*}
\tilde{l}^M_k=\frac{N}{\omega}\lambda_k e^{-k\delta_n}, \;\forall k\leq n^{2/5}.
\end{align*}
In order to give estimates to \eqref{eqminorupfunction}, we need further detail on the corresponding $\delta_n$, giving suitable approximation to it. Notice that we can give the following bounds
\begin{align*}
-(nb_n)^{2/3}M\leq & \sum_{k=1}^\infty k\lambda_ke^{-k\delta_n}-\frac{N}{\omega}\leq -(nb_n)^{2/3}M+\sum_{k\geq n^{2/5}}k\lambda_ke^{-k\delta_n},
%-(nb_n)^{2/3}M\leq &\sum_{k=1}^\infty k\lambda_ke^{-k\delta_n}-\frac{N}{\omega}\leq -(nb_n)^{2/3}M+O(n^{4/5})\\
\end{align*}
and we also know that $\sum_{k\geq n^{2/5}}k\lambda_ke^{-k\delta_n}=O(n^{4/5})$.
%\lu{A priori bounds on $\delta$? We need to ensure that $\delta_n$ is positive (or at least bounded from below uniformly) before the previous argument, that then is used to find some estimate for $\delta_n$.}
Thus 
\[\frac{N}{\omega}\sum_{k=1}^\infty k\lambda_ke^{-\delta_n k}%=\frac N \omega -M(nb_n)^{2/3}+O(n^{4/5})
=\frac 1 \omega N\nu\] 
where we set
$\nu:=1-M\left(\frac{b_n^2}{n}\right)^{1/3}+O(n^{-1/5})$. Relation \eqref{eqseries} then tells us that $e^{-(1+\delta_n)}=\nu e^{-\nu}$, that means
\begin{align*}
\delta_n=-\log(\nu)-1+\nu=\frac 1 2 M^2\left(\frac{b_n^2}{n}\right)^{2/3}+o(M^2\left(\frac{b_n^2}{n}\right)^{2/3}).
\end{align*}
This estimate of $\delta_n$ allows us to give the estimate  we were looking for, by plugging in $\tilde{l}^M$ in \eqref{eqminorupfunction}. 
\begin{align*}
    -\sum_{k\leq n^{2/5} } l^M_k&\log(\frac{\omega  l^M_k}{N\lambda_ke})\leq\frac N \omega\sum_{k} \lambda_ke^{-k\delta_n}+\delta_n\frac{N}{\omega}\nu+o((1+M^2)b_n^2)\\
    &=\frac{N}{\omega}\nu(1-\frac{\nu}{2})+\frac{N}{\omega}\left(-\log(\nu)-1+\nu\right)\nu+o((1+M^2)b_n^2)\\
    &=\frac{N}{\omega}\nu\left(\frac{\nu}{2}-\log(\nu)\right)+o((1+M^2)b_n^2)\\
    &=\frac{N}{\omega}\left( \frac{1}{2}-\sum_{k\geq 3}\frac{1}{k(k-1)}\left(M\left(\frac{b_n^2}{n}\right)^{1/3}+O(n^{-1/5})\right)^k\right)+o((1+M^2)b_n^2)\\
\end{align*}
To obtain the estimation above, we further use \eqref{eqseries2}, we expand $\nu(\frac{\nu}{2}-\log\nu)$ and we use the approximation of $\nu$. %\lu{I added a comment like this, just to give a short explanation. Does it make sense?}
%\marginnotemx{ expand the $\nu(\frac{\nu}{2}-\log\nu)$ ? instead? }
%\marginnotemx{I see, the redaction there is quite elliptic, I expanded the expression $\nu\left(\frac{\nu}{2}-\log(\nu)\right)$ into a power series and then replace the term $\nu$  by it approximation  $1-M\left(\frac{b_n^2}{n}\right)^{1/3}+O(n^{-1/5})$. Is it fine?}
Having now a convincing estimate of the leading terms in equation \eqref{eqmicroupsumj} that depends on $M$, we can turn to the estimation of the optimal value for $M$. Specifically, we want to prove that the optimal $M$ is upper bounded uniformly in $n$. 
First, let us define the map $F_n(\cdot)$ which associates to $M$ the upper bound for $ -\sum_{k\leq n^{2/5} } l^M_k\log(\frac{\omega  l^M_k}{N\lambda_ke})$ that we have found above:
%function that we want to maximize in \eqref{eqmicroupsumj}
%be the map defined from the sum of the term above with $\epsilon^2Mb_n^2$ as follows :
$$
F_n(M):=\frac{N}{\omega}\left( \frac{1}{2}-\sum_{k\geq 3}\frac{1}{k(k-1)}\left(M\left(\frac{b_n^2}{n}\right)^{1/3}+O(n^{-1/5})\right)^k\right)+M\epsilon^2b_n^2+o(b_n^2).
$$
From the expression of the sum in the formula above, we deduce that there is $\eta>0$ such that for any $n\in \N$ large enough and any $M\geq \eta$ , $M\left(\frac{b_n^2}{n}\right)^{1/3}+O(n^{-1/5})\geq \frac 1 2 M\left(\frac{b_n^2}{n}\right)^{1/3}$, and thus
\begin{align*}
    F_n(M)-\frac{N}{2\omega}\leq -\frac{1}{48}M^3b_n^2+\epsilon^2 Mb_n^2 \leq 0,
\end{align*}
where the upper bound is obtained by picking the term in the sum above for 
 $k=3$. 
%\lu{I do not understand how you get to the inequality above. If we lower bound the sum with its value in $k=3$, I do not understand the prefactor and where $n$ is. Maybe is something differente, then. Can you please briefly explain? Thanks!}\marginnotemx{The prefactor is voluntarily lowered to make $n^{-1/5}$ disappear. Do you agree that for $n$ large enough and $M>0$ fixed $\frac 1 2M\left(\frac{b_n^2}{n}\right)^{1/3}\geq O(n^{-1/5})$? I just replace things by $\frac 1 2M\left(\frac{b_n^2}{n}\right)^{1/3}$ in the expansion then keeping only the first index $k=3$. (May be I took a prefactor even too small.}

Notice in addition that $(F_n(\cdot)-\frac{N}{2\omega})b_n^{-2}$ converges (uniformly on each compact interval) to the map $M\mapsto -\frac{M^3}{6}$. Thus,  according to Proposition~\ref{propunicon}, the sequence of elements $D_n:= argmax_{M\in K_n\cap \mathbb R_+}\{F_n(M)\}$ maximizing $F_n$ tends to $0$ and there is some $M_0>0$ such that $F_n(\cdot)$ attains its maximum on $\left(-\left(\frac{n}{b_n^2}\right)^{1/3}(1-\frac{1}{\omega}),M_0\right)$ for $n$ large enough. Having an estimate on the optimal size %$D_n$ which drives the size 
of the term $\sum_{k\geq n^{2/5}}\frac{k^3}{n^2}l_k^M$ and we are now able to conclude our estimation of the probability $\PP(L^n=l)$ from equation \eqref{eqmicroupsumj}  (we recall that $a=\log(\omega)-\frac{\omega}{2}$) :

{\small \begin{align}
   \log(\PP(L^n=(l_k)_{k\in \N}))&\leq aN+\sup_{M\in K_n\cap (-\infty,M_0] }\left(2\epsilon^2 (M+\theta)b_n^2-\sum_{1\leq k\leq \epsilon(nb_n)^{2/3}} l^M_k\log(\frac{\omega l^M_k}{N\lambda_ke})\right)\nonumber\\
  % &\leq aN+2\epsilon^2 \theta b_n^2-\sum_{1\leq k\leq \epsilon(nb_n)^{2/3}} l^M_k\log(\frac{\omega l^M_k}{N\lambda_ke})\nonumber\\
   &\leq 2\epsilon^2(M_0+\theta) b_n^2+(\log(\omega)-\frac{\omega}{2})N+\frac{N}{2\omega}\label{eqmicroupfinal}\\
   &\leq  2\epsilon^2 (M_0+\theta) b_n^2-\frac{1}{6}\theta^3b_n^2,\nonumber
\end{align}}where the passage to the line \eqref{eqmicroupfinal} is done by upper bounding the sum with the direct maximum of each term $l_k\mapsto l_k\log(\frac{\omega l_k}{N\lambda_ke})$ involved in the sum, i.e.  whenever $l=\frac{N}{\omega}\lambda_k$. We deduce from \eqref{eqseries2} the maximum value of the overall sum. This concludes the lemma on the upperbound : for any $\epsilon>0$ and any $l\in E_{\alpha_n}$,
\begin{align*}
    \lim_{n\rightarrow\infty}\frac{1}{b_n^2}\log(\PP(L^n=(l_k)_{k\in \N}))\leq  2\epsilon^2 \theta -\frac{1}{6}\theta^3
\end{align*}
This ends the proof.

\section{Remaining proofs and remarks}\label{appendix}
\subsection{Macroscopic and Mesoscopic estimates: proof of Lemma~\ref{lem:macro} and Lemma~\ref{lem:meso}}\label{sec:lemmas_meso_macro}
We begin by proving Lemma~\ref{lem:macro}.

\begin{proof}[Proof of Lemma~\ref{lem:macro}]
We first prove that $\log(\PP_{n,p}(L^n=l^n))\ll -b_n^2$ whenever the the term $l^n$ contains non zero coordinates $l_k^n$ for $k> \beta_n$. One way to see that is to fix $\epsilon >0$ and notice that according to the lemma 3.3, 3.4 and 3.5 of \cite{LA21},  there is $h_\epsilon >0$ such that for  any $l$ containing non zero coordinate $l_k$ for some $k \geq \rho n$ and any $n$ large enough,
\begin{align*}
     \log(\PP_{n,p}(L^n=l^n))\leq  -h_\rho n.
\end{align*}
Thus we now assume that $l^n_k=0$ for any $k\geq \rho n$.  Recall from the equation \eqref{eqprimord0} that the total contribution can be expressed as 
\begin{align*}
    \PP_{n,p}(L^n=l^n)=e^{o(b_n^2)} G{Ma}(l)G_{Me}(l),
\end{align*}
with 
\begin{align*}
    G_{Ma}(l)=\prod_{k=\beta_n}^n\left( \frac n e \right)^{kl_k}z_k^n(1)^{l_k}\, \textit{and }G_{Me}(l)=\prod_{k=1}^{\beta_n}\left( \frac n e \right)^{kl_k}z_k^n(l_k),
\end{align*}
where here we distinguish simply between \textit{macroscopic} and \textit{non-macroscopic} components.
Whenever $k\geq \beta_n$, the term $z_k^n(1)$ can be upper bounded using the following estimate of  the probability $\Pp_{k,n}$ of having a connected graph from \cite{Ste70} :
\begin{align*}
    \log\left(\Pp_{k,n}\right)&\leq -(k-1)\log(1-e^{k\log(1-p)})\\
    &\leq -(k-1)\log(\frac{k}{n})+O(\frac{k}{n}).
\end{align*}
When $k\leq \beta_n$ we rather use the estimate of $\Pp_{k,n}$ computed in Proposition~\ref{lemgraphconnex}.
Then notice that using Stirling approximation to expand the term $k!$ in the expression of $z_k^n(1)$, for $n$ large enough,
\begin{align*}
    \log(G_{Ma}(l))&\leq\sum_{k=\beta_n}^n l_k \left(\log(\frac{n}{k}) + \frac{k}{2}(n-k)\log(1-p)\right)+o(b_n^2) \\
    &\leq \sum_{k=\beta_n}^n l_k \left(\log(\frac{n}{k}) -\frac{k}{2}(n-k)\frac{1+\theta\left(\frac{b_n^2}{n}\right)^{1/3}}{n}\right)+o(b_n^2).
\end{align*}
On the other hand applying Stirling formula on $k!$ and $l_k!$ in the contribution $\tilde F_{Me}(l)$ gives
\begin{align*}
     \log(G_{Me}(l))\leq &\sum_{k=1}^{\beta_n} l_k\left(\log(\frac{n\lambda_k}{l_k}) + \frac{k}{2}(n-k)\log(1-p) \right)+\frac{1}{24n^2}\sum_{k=n^{2/5}}^{\beta_n}k^3l_k\\
     &+C\sum_{N\geq k\geq n^{2/5}}l_k\left(k\left(\frac{k}{2n}\right)^2\left(\frac{b_n^2}{n}\right)^{1/3}+k\left(\frac{k}{2n}\right)^3+\log(k)\right) +o(b_n^2)\\
     \leq &\sum_{k=1}^{\beta_n} l_k\left(\log(\frac{n\lambda_k}{l_k}) + \frac{k}{2}+k\log\left(1+\theta \left(\frac{b_n^2}{n}\right)^{1/3}\right) \right) +\frac{1}{12n^2}\sum_{k=n^{2/5}}^{\beta_n}k^3l_k+o(b_n^2)
\end{align*}

Let $M=\sum_{k=1}^{\beta_n} kl_k$ and suppose $M\leq n-\beta_n$ ($l^n$ has at least one coordinate $l_k>0$ for $k\geq \beta_n$). Using the Lagrange multiplier technique, the sequence $(l_k)_{k\leq \beta_n}$ maximizing $\tilde F_{Me}(l)$ and $M=\sum_{k=1}^{\beta_n} kl_k$ satisfies $l_k=n\lambda_ke^{-1-\delta_nk}$.  With $\delta_n \leq 0$ if and only if $M \geq \frac{n}{e}$.  Furthermore this condition imposes that $n\lambda_{\beta_n}e^{-1-\delta_n \beta_n}\leq M$ and thus $-\delta_n \leq \frac 1 {\beta_n}\left(1+\frac 5 2 \log(n)\right)$. We recall from equation \eqref{eqseries2} that $\sum \lambda_k=\frac 1 2$ and thus
$$
\sum_{k=1}^{\beta_n} l_k\log(\frac{n\lambda_k}{l_k})=\sum_{k=1}^{\beta_n} \lambda_ke^{-1-k\delta_n}+\delta_nM \leq \frac M 2 +\delta_n M.
$$
Then we deduce from the size of $\delta_n$ the following upper estimates for $n$ large enough,
\begin{align*}
    \log(G_{Ma}(l))&\leq -\frac 1 2 (n-M)-\rho (n-M)+o(b_n^2) \leq -\frac 1 4 \beta_n \\
    \log(G_{Me}(l))&\leq \delta_n M+\frac 1 {12}\left(\frac{\beta_n}{n}\right)^{1/3}M+M\theta \left(\frac{b_n^2}{n}\right)^{1/3}+ O(b_n^2)\\
    &\leq \theta (nb_n)^{2/3} +o(b_n^2).
\end{align*}

This implies the following decay for the total probability when $n$ is large enough,
\begin{align*}
\log\left(\PP_{n,p}(L^n=l^n) \right)\leq -\frac 1 {8} \beta_n.    
\end{align*}
    
\end{proof}

We prove here Lemma~\ref{lem:meso}. %\lu{what are we referring to, here???}% from page \pageref{lem:meso} :
\begin{proof}[Proof of Lemma~\ref{lem:meso}]%\label{lem:meso}
Recall from the text after \eqref{eq:z_k} that $\alpha_n=\epsilon(nb_n)^{2/3}$ and $\beta_n= n^{17/24}b_n^{8/12}$.  
%We need to apply Proposition~\ref{lemgraphconnex} to obtain the following estimate for each of the term $z_k^n(1)$ in the product of the contributions of the mesoscopic components in $F_{Me}(l)$ {\color{blue}whenever $\sum_{k\geq \beta_n}kl_k=0$} :
Recall that $F_{Me}(l)$ satisfies the following expansion
\begin{align*}
     F_{Me}(l)  &=  \left(\frac{n}{n-\sum_{k=\alpha_n}^nkl_k}\right)^{n-\sum_{k=\alpha_n}^nkl_k}\\
 & \quad\prod_{k=\alpha_n}^{\beta_n} \left(\frac{n}{e}\right)^{kl_k}(z_{k}^n(1))^{l_k}(1-p)^{\frac{1}{2}(\sum_{i=1}^{\alpha_n}il_i)kl_k} ;\notag\\
\end{align*}
with
\begin{align*}
z_k^n(1)=\frac {\PP_{k,p}(L^k_k=1)(1-p)^{\frac12(n-k)k}}{k!}.    
\end{align*}
The aim is to provide computations and estimates of $F_{Me}(l)$ for some $l\in \mathcal N_n$ that satisfies $l_k=0$ for any $k\geq \beta_n$ for $n$ large enough. 
%and such that
%$\sum_{k=1}^n kl_k \leq M(nb_n)^{2/3}$ ( and thus $\sum_{k=\beta_n}^nkl_k=0$ for $n$ large enough) and  
%$\Me_n(l^n)\to \mu$.
To do this, we expand the three following terms appearing in $\log(F_{Me}(l))$. First,
\begin{align*}
&\left(n-\sum_{k=\alpha_n}^nkl_k\right)\log\left( \frac{n}{n-\sum_{k=\alpha_n}^nkl_k}\right)\\
&=-\left(n-\sum_{k=\alpha_n}^nkl_k\right)\log(1-\frac{\sum_{k=\alpha_n}^nkl_k}{n})\\
&=\sum_{k=\alpha_n}^nkl_k-\frac{\left(\sum_{k=\alpha_n}^nkl_k\right)^2}{2n}-\frac{\left(\sum_{k=\alpha_n}^nkl_k\right)^3}{6n^2}-\sum_{k=3}^\infty\frac{1}{k(k+1)}\frac{\left(\sum_{k=\alpha_n}^nkl_k\right)^{k+1}}{n^k}.
\end{align*}%
Second,
\begin{align*}%
   \frac{1}{2}&(n-\sum_{k=\alpha_n}^nkl_k)\left(\sum_{k=\alpha_n}^{\beta_n}kl_k\right) \log\left(1-p\right)\\
   &=-\frac {1+\theta\left(\frac{b_n^2}{n}\right)^{1/3}} 2 \sum_{k=\alpha_n}^nkl_k+\frac 1 {2n}\left(1+\theta\left(\frac{b_n^2}{n}\right)^{1/3}\right)\left(\sum_{k=\alpha_n}^nkl_k\right)^2+o(b_n^2).
\end{align*}
Third,
\begin{align*}%
 &\frac{1}{2}\sum_{k=\alpha_n}^{\beta_n}(n-k)kl_k\log\left(1-p\right)\\
 &=-\frac {1+\theta\left(\frac{b_n^2}{n}\right)^{1/3}} 2 \sum_{k=\alpha_n}^{\beta_n}kl_k +\frac 1 {2n}\left(1+\theta\left(\frac{b_n^2}{n}\right)^{1/3}\right)\sum_{k=\alpha_n}^{\beta_n}k^2l_k+o(b_n^2).
\end{align*}%
We gather them with the estimate of $\log(\PP_{k,p}(L^k_k=1))$ obtained in the point (i) of Proposition~\ref{lemgraphconnex}. We also emphasize that the quantity $C_n(k)$ appearing in the point (i) of Proposition~\ref{lemgraphconnex} can be summed as follows : $\sum_{k=\alpha_n}^{\beta_n}l_kC_n(k)=\frac{1}{24n^2}\sum_{k=\alpha_n}^{\beta_n}k^3l_k+o\left(b_n^2+\frac{1}{n^2}\sum_{k=\alpha_n}^{\beta_n}k^3l_k\left(\frac{b_n^2}{n}\right)^{1/8}\right)$. 
These three expansions put together along with equation \eqref{eq:prob_conn_regime_meso} provide the following estimate:

\begin{align}
&\log\left(F_{Me}(l)\right)=\frac{(\sum_{k=\alpha_n}^{\beta_n}kl_k)^2}{2n}\theta \left(\frac{b_n^2}{n}\right)^{1/3}-\frac 1 2\sum_{k=\alpha_n}^{\beta_n} kl_k\theta^2\left(\frac{b_n^2}{n}\right)^{2/3}+\frac 1 3\sum_{k=\alpha_n}^{\beta_n} kl_k\theta^3\frac{b_n^2}{n}\nonumber\\
&-\frac{\left(\sum_{k=\alpha_n}^{\beta_n}kl_k\right)^3}{6n^2}-\sum_{j=3}^\infty \frac{1}{j(j+1)}\frac{1}{n^j}\left(\sum_{k=\alpha_n}^{\beta_n} kl_k\right)^{j+1}+ \frac{1}{24n^2}\sum_{k=\alpha_n}^{\beta_n}k^3l_k\label{eqmainsumlem}\\
&+o\left(b_n^2+\frac{1}{n^2}\sum_{k=\alpha_n}^{\beta_n}k^3l_k\left(\frac{b_n^2}{n}\right)^{1/8}\right)\\
&\leq -\frac 1 6 b_n^2\left(\int_\epsilon^CxdMe_n(l)-\theta \right)^3+\frac{\theta^3}{6}b_n^2+\frac{b_n^2}{24}\int_\epsilon^\infty x^3dMe_n(l)+\theta b_n^2\\&+o\left(\left(1+\int_\epsilon^\infty xdMe_n(l)\right) b_n^2\right),\notag
\end{align}%
where we applied the rough estimate $\sum_{k=\alpha_n}^{\beta_n}kl_k\leq n$ and thus $\frac 1 {6n}\sum_{k=\alpha_n}^{\beta_n}kl_k\theta^3b_n^2\leq \theta^3b_n^2$. In the particular case where $\int_{\epsilon}^\infty x d \Me_n(l^n)(x)\geq 3\theta$ we get that for some $h>0$,
%the first conclusion of lemma \ref{lem:meso}:
%
\begin{align*}
   \frac{1}{b_n^2}\log\left(F_{Me}(l)\right)\leq -h\left(\int_{\epsilon}^C x d \Me_n(l^n)(x)\right)^3.
\end{align*}
This proves part \textit{i)} of the lemma.\\

To prove part \textit{ii)} of the lemma, notice that $\sum_{k=\alpha_n}^{\beta_n}kl_k\leq (nb_n)^{2/3}\int_\epsilon^\infty xdMe_n(l) $ with the latter integral bounded by assumption by a constant $M>0$. Thus $l_k=0$ for any $k\geq M(nb_n)^{2/3}$ and the terms $\frac 1 n\sum_{k=\alpha_n}^{\beta_n}kl_k\theta^3b_n^2$ and $\sum_{j=3}^\infty \frac{1}{j(j+1)}\frac{1}{n^j}\left(\sum_{k=\alpha_n}^{\beta_n} kl_k\right)^{j+1}$ in line \eqref{eqmainsumlem} are negligible whereas the term $\sum_{k=\alpha_n}^{\beta_n}k^3l_k=O(b_n^2)$. Thus,
\begin{align}
\log&\left(F_{Me}(l)\right)\notag\\
&=-\frac 1 6 b_n^2\left(\frac{1}{(nb_n)^{2/3}}\sum_{k=\alpha_n}^{\beta_n}kl_k-\theta \right)^3-\frac{\theta^3}{6}b_n^2+\frac 1 n\sum_{k=\alpha_n}^{\beta_n}kl_k\theta^3b_n^2\notag\\
&\quad+\frac{b_n^2}{24}\sum_{k=\alpha_n}^{\beta_n}\left(\frac{k}{(nb_n)^{2/3}}\right)^3l_k+o(b_n^2)\label{eqmesoprincip2}\\
&=-\frac 1 6 b_n^2\left(\int_\epsilon^CxdMe_n(l)-\theta \right)^3-\frac{\theta^3}{6}b_n^2\notag\\
&\quad+\frac{b_n^2}{24}\int_\epsilon^Cx^3dMe_n(l)+o(b_n^2).
\label{eqmesoprincip}
\end{align}
%
%Notice that to pass from line \eqref{eqmesoprincip2} to line \eqref{eqmesoprincip} we remove the term $\frac 1 n\sum_{k=\alpha_n}^{\beta_n}kl_k\theta^3b_n^2$ which is in $o(b_n^2)$ since $\sum_{k=\alpha_n}^{\beta_n}kl_k\leq M(nb_n)^{2/3}$.
Once we take the limit in equation \eqref{eqmesoprincip}, we obtain the requested conclusion of the second statement of lemma \ref{lem:meso}, i.e. 
% \lim_{\substack{\epsilon\to 0, \\C\to\infty}} 
\begin{equation}
       \lim_{n\rightarrow \infty}\frac{1}{b_n^2}\log(F_{Me}(l)) =-\frac 1 6 \left(\int_\epsilon^C xd\mu-\theta \right)^3-\frac{\theta^3}{6}+\frac{1}{24}\int_\epsilon^C x^3 d\mu.
\end{equation}

\end{proof}
\subsection{Cardinality estimates}

We see that, in the proof of the upper bound in Theorem~\ref{thmweakldp}, it is important to control the cardinality of the set 
\[
\{l\colon \Me_n(l)\in B_{\delta}(\mu)\}\subseteq \Ncal_n,
\]
for any $\mu\in\Mspace$ and $\delta>0.$ Here we give estimates for this cardinality.
\begin{lemma}
    For every $\mu\in\Mspace$ and every $\delta>0$, under Hypothesis~\ref{hypbn}, we have that
    \[
    |\{l\in\Ncal_n\colon \Me_n(l)\in B_{\delta}(\mu)\}|=e^{o(b_n^2)}.
    \]
\end{lemma}

\begin{proof}
    To prove the upper bound over the whole event $\mathcal N_n$ we thus only need to look at its cardinality :\\
for this we proceed as in \cite{LA21}, for any $l\in \mathcal N_n$, denote $H(l):=\{k, l_k>0\}$ and notice that $|H(l)|\leq 2n^{1/2}$.
\begin{align*}
n=\sum_{k\in H(l)}kl_k\geq \sum_{k=1}^{|H(l)|}k\geq \frac{|H(l)|(|H(l)|-1)}{2}.
\end{align*}
Thus 
\begin{align*}
|\mathcal N_n|&\leq |\{ l, \sum_k kl_k = n\}|\\
&\leq \sum_{H\subset \{1,\dots,n\},|H|\leq 2n^{1/2}} |\{ (l_k)_{k\in H}\in \N^{|H|}, \sum_k kl_k = n\}|\\
&\leq \sum_{H\subset \{1,\dots,n\},|H|\leq 2n^{1/2}} |\{ (l_k)_{k\in H}\in \N^{|H|}, \sum_k l_k = n\}|\\
&\leq \sum_{k=1}^{2n^{1/2}}{n \choose k}{n+k \choose k}=e^{O(n^{1/2}\log(n))}=e^{o(b_n^2)}.
\end{align*}
\end{proof}

\subsection{Proof of Proposition~\ref{propunicon}}\label{appa}
Fix $\eta>0$, the closeness of $C(X,\mathbb R)$ with respect to the uniform norm $\|.\|_{\infty}$ implies that $f$ is continuous and since $x$ is it unique maximum, then there is $\epsilon >0$ such that any element $y\in \bar B_\epsilon(x)^c$, outside the close (and thus compact) ball satisfies :
$$f(y)\leq f(x)-\epsilon.$$
By the uniform continuity of $(f_n)_{n\in \N}$, there is $n_0\in \N$ such that for any $n\geq n_0$ and any $y\in \bar B_\epsilon(x)^c$,
   \begin{align*}
   \left\{
   \begin{array}{r c l}
    &f_n(y)\leq f(x)-\frac{3\epsilon}{4} \\
        & f_n(x)\geq f(x)-\frac{\epsilon}{4}.
   \end{array}\right.
   \end{align*}
   Thus for any $\eta>0$ and any $n\in \N$ large enough $argmax f_n \subset B_\eta(x)$, which prove the convergence to the element $x$ of any sequence $(x_n)_{n\in\N}$ such that for any $n\in \N$, $x_n\in argmax f_n$.

\subsection{Topology on $\M_{\N}(0,\infty)$, properties of $I(\cdot)$ and comparison with \cite{Pu05}}\label{sectopomn}
We provide here some classical results on the topology of $\M_{\N}(0,\infty)$ and the regularity of the rate function $I(.)$. We compare at the end of this section the embedding space $\M_{\N}(0,\infty)$ with the one chosen in \cite{Pu05} to model mesoscopic components on a random graph.\\

The vague topology on $\M_\N(0,\infty)$ is induced by the distance $dm$ defined as follows. For $\mu,\nu \in \M_{\N}(0,\infty)$, let $\mu:=\sum_{i\in \Z}\delta_{u_i}$ and $\nu:=\sum_{i \in \Z}\delta_{v_i}$ with $u,v\in \overline{ \mathbb R}_+^\Z$. 
The distance $dm$ is defined as
\begin{align*}
dm(\mu,\nu)=\min\{d(u,v),d(u,(v_{i+1})_{i\in \Z}),d((u_{i+1})_{i\in \Z},v)\}.
\end{align*}%
where $ d$ is the distance 
\begin{align}
 d (u,v)=\sum_{i< 0}\frac{1}{2^{|i|}}|u_i-v_i|+\sum_{i\geq 0}\frac{1}{2^{|i|}}|e^{-u_i}-e^{-v_i}|.
\end{align}%
%
%\lu{I do not understand this distance. What happens when $u$ has a finite number of non-zero entries? The sum for $i>0$ is not well defined, right? Moreover, $i=0$ belongs to both sums, I guess there is a typo somewhere.}\marginnotemx{ actually the sequence is two vector/sequence, one with terms $u_i$ below one and the other containing terms $u_i$ above $1$. In each case, we complete each into full sequence by appending terms $0$ for the first one and $\infty$ for the second. For the term in $i=0$ it is a typo! \lu{Sorry, I've seen the proper definition later and I see. Thanks for the clarification! As concerning the typos, can you please correct it when you see them? There are many typos, it is sometimes very hard for me, when I read, to understand what you mean and I do not know what I should correct.}}

We identify the following compact sets in $\M_\N(0,\infty)$ equipped with the vague topology.
%In what follow, we prove some convenient topological features from Puhalskii's article \cite{Pu05} adapted to our topological space $\M_\N(0,\infty)$ and our rate function $I(\cdot)$.
\begin{proposition}\label{prop:compact_sets}
For any $M>0$ the set \[A_M=\{\mu\in \M_\N(0,\infty),\int xd\mu(x)\leq M\}\]
is compact in $\M_\N(0,\infty)$ for the vague topology. 
%\lu{what do you mean by \emph{bounded}? I do not think that the set $A_M$ is bounded, in the sense that its element are not bounded measures on $(0,\infty)$.}\marginnotemx{sorry I was thinking of bounded expectancy.}
\end{proposition}

\begin{proof}
Notice that the convergence on $\M_\N(0,\infty)$ of a sequence $\mu_n=\sum_{i \in \Z} \delta_{u_{i,n}} \in \M_\N(0,\infty)$ to $\mu=\sum_i \delta_{u_i}$  for the vague topology corresponds to the simple convergence of the sequence $(u_{i,n})_{i\in \Z}$ to $(u_i)_{i\in \Z}$ in 
$[0,\infty]$. Thus from a sequence made of terms $\mu_n=\sum_{i \in \Z} \delta_{u_{i,n}}$ one can extract by a diagonalisation argument a sub-sequence $(\mu_{n_k})_{k\in \N}$ such that $u_i:=\lim_{k\rightarrow \infty} u_i^{n_k}$ is well defined for any $i \in \Z$ and satisfies the same 
relative order $0 \dots \leq u_{-1}\leq u_0\leq u_1\leq u_2\leq \dots$. To prove that the measure  $\mu:=\sum_{i\in \Z}\delta_{u_i}$ is then well defined as an element of $\M_\N(0,\infty)$, it remains to check that it is $\sigma$-finite on $(0,\infty)$ which in our case means that for any $C>0$, 
\begin{align}\label{eqcondsigmfini}
\sum_{i, \frac 1 C \leq u_i \leq C}u_i < \infty.
\end{align}
In particular, if for some $M>0$,
This condition \eqref{eqcondsigmfini} is satisfied since $\mu_n\in A_M$ for any $n\in \N$, and
since $\mu_n\in A_M$, Fatou's lemma ensures that 
\begin{align*}
     \int xd\mu(x)=\sum_{i\in \Z}u_i\leq \liminf \int xd\mu_{n_k}\leq M,
\end{align*}
and condition \eqref{eqcondsigmfini} is satisfied and  $\mu \in A_M$ proving its compactness. 
\end{proof}

%Puhalskii From \cite{Pu05} already proved the lower semi-continuity of $I\circ F$. \lu{What does this sentence mean?}\gm{Indeed, please fix}\marginnotemx{sorry I don't get what is wrong with the sentence}
%\marginnotemx{I changed it, the rate function in Puhalskii is the same just pulled back on another space.\lu{the sentence is still incomplete and unclear.}}
We now prove that the rate function $I(\cdot)$ is  a good rate function, i.e., it is lower semi-continuous and its sublevel sets are compact. %Note that in~\cite{Pu05} the lower semi-continuity of $I\circ F$ is already proved.
\begin{proposition}\label{propsemicon}
    The map $I(\cdot)$ is lower semi-continuous on $\M_\N(0,\infty)$ and is a good rate function.
\end{proposition}

\begin{proof}

We only have to prove that $I(\cdot)$ is lower lower semi-continuous , indeed, since  $I^{-1}([-\infty,a])$ is included in the compact set $A_M=\{\mu\in \M_\N(0,\infty),\int xd\mu(x)\leq M\}$ for some $M>0$, the map $I(\cdot)$ would then be a good rate function. To prove the lower semi continuity of $I(\cdot)$,
let $\mu_n=\sum_{i \in \Z} \delta_{u_{i,n}} \in \M_\N(0,\infty)$ and $\mu=\sum_i \delta_{u_i}\in \M_\N(0,\infty)$ such that $\mu_n \underset{n\rightarrow \infty}\rightarrow \mu$. If $I(\mu)=\infty$ then from Fatou's lemma we get $\liminf_{n\rightarrow\infty} \sum_{i\in \Z} u_{i,n} \geq \sum_{i \in \Z} u_i$ 
and thus from the explicit formula of $I(\cdot)$, $\liminf_{n\rightarrow\infty} I(\mu_n)=I(\mu)=\infty$.
Else, suppose $I(\mu)<\infty$ and $I(\mu_n)<\infty$ for any $n$ (the claim follows immediately otherwise). Then there is some 
$M>0$ such that $\liminf \int xd\mu_n\leq M$ ( otherwise we get from the previous case $\liminf_{n\rightarrow\infty} I(\mu_n)=\infty$). 

Notice in one hand that $\nu \mapsto \int x^3d\nu$ is continuous on the set $A_M:=\{(u_i)_{i\in \Z}, \sum_{i\in \Z}u_i\leq M \}$,  we also prove that $I(\cdot)$
%\gm{Please do a search+replace from $I(\cdot)$ to $I(\cdot)$} 
is the sum of the continuous function $\nu \rightarrow -\frac{1}{24}\int x^3d\nu+\frac{\theta^3}{6}$  with the lower semi continuous one 
and on the other hand the following map is lower semi-continuous on $A_M$
$$
J:\nu \mapsto \frac{1}{6}(\int x d\nu-\theta)^31_{(\int x d\nu\geq \theta)}.
$$ 
Indeed by Fatou lemma's, for any sequence $\nu_n$ of elements of $A_M$ converging to $\nu\in A_M$, $\liminf_{n\rightarrow\infty}\int xd\nu_n \geq \int xd\mu$ and since $x \mapsto x1_{x\geq 0}$ is a nonegative continuous map, we get by composition that the map $\nu \mapsto \frac{1}{6}(\int xd\nu-\theta)^31_{(\int x d\nu \geq \theta)}$ is lower semi-continuous. Thus $I(\cdot)$ is semicontinuous as the sum of $J$ with the continuous map $\nu \rightarrow -\frac{1}{24}\int x^3d\nu+\frac{\theta^3}{6}$, which proves our assumption.
\end{proof}

%\lu{put here: the definition of the metric that metrizes the topology we are interested in, what are compact sets in this topology, if possible also relation with the space and the topology Puhalskii uses in its work (are they equivalent?)} %\marginnotemx{ I think actually any bounded set is compact, by the way $I(\mu)$ is a good rate function. \lu{Good, let's write a short proof of it. It is important also to draw links and underline differences with the topological setting in puhalskii.}}

We give here a short overview with comparisons between the topology used in \cite{Pu05} and the one we introduced.
In  \cite{Pu05} Puhalskii identifies the sizes of connected components of a realisation of the  \ER- graph $\mathcal G(n,p)$ with an element in 
$\breve{\mathbb S}:=\{(u_i)_{i\geq 0}\in \mathbb R_+^\N, u_i\geq u_{i+1},\}$. The sequence $((nb_n)^{2/3}u_i)_{i\in \N}$  would then represent the set of connected components of the graph in decreasing order. Puhalskii's statement in \cite[Theorem 2.4]{Pu05} is a LDP satisfied by the connected component of a graph as represented on the topological space $\breve {\mathbb S}$ with the product topology (topology with respect to componentwise convergence).  
This topological space  $\breve {\mathbb S}$ can be embedded into a subset of $\overline{\M_{\N}(0,\infty)}$  through the following map $F: \breve{\mathbb S}\mapsto \overline{\M_{\N}(0,\infty)}$ defined by
$$
F(u):=\sum_{i\geq 0}\delta_{u_i},
$$
where  $\overline{\M_{\N}(0,\infty)}$ is the completion of $\M_{\N}(0,\infty)$ for the distance $dm$.
Now we prove that theorem \ref{thmweakldp} can be extended to  $\overline{\M_{\N}(0,\infty)}$ and thus comprises the statement  \cite[Theorem 2.4]{Pu05}  whenever $(b_n)_{n\in \N}$ satisfies Hypothesis \ref{hypbn}.

\begin{lemma}\label{lemcomplete}
The law of the random variable $\Me_n(L^n)$ under $\mathbb{P}_{n,p_n(\theta)}$ %\lu{Are we using the notation $\mathbb{P}_{n,p_n(\theta)}$ or $\mathbb{P}_{n,p_n(\theta)}$?!}
satisfies a \textit{large deviation principle} with speed $b_n^2$ and rate function $I :\overline{\M_\N(0,\infty)} \mapsto \bar{\mathbb{R}}_+$ defined by
%\label{eq:rate_function}
\begin{equation*}
    I(\mu):=-\frac{1}{24}\int_0^\infty x^3d\mu+\frac{1}{6}(\int_0^\infty xd\mu-\theta)^31_{(\int_0^\infty xd\mu \geq \theta)}+\frac{\theta^3}{6},
\end{equation*}
%\sum_{i\in \Z} u_i^3+\frac{1}{6}(\sum_{i\in\Z} u_i-\theta)^31_{(\sum_{i\in\Z} u_i\geq \theta)}+\frac{\theta^3}{6}
whenever the integral $\int_0^\infty xd\mu$ is finite and $I(\mu)=\infty$ otherwise.    
\end{lemma}

\begin{proof}
To prove it we only need to show that $\int xd\mu=\infty$ for any $\mu \in \overline{\M_{\N}(0,\infty)}\backslash \M_{\N}(0,\infty)$. Actually we need to prove a bit more, we will prove that
\begin{align}\label{eqmominf}
\lim_{\delta \to 0} \inf_{\nu \in B_\delta(\mu)}\int xd\nu =+\infty.    
\end{align}
Then $I(.)$ still extends onto  $\overline{\M_{\N}(0,\infty)}$ as a good rate function.  Lemma \ref{lem:meso} and  Lemma~\ref{lem:macro} ensure that $\lim_{\delta \to 0}\frac{1}{b_n^2}\log\left(\PP_{n,p}(L^N\in B_\delta(\mu))\right)=-\infty=-I(\mu)$  which completes the proof of Lemma \ref{lemcomplete}. \\
We now prove the limit \eqref{eqmominf}. Let  $\mu \in \overline{\M_{\N}(0,\infty)}\backslash \M_{\N}(0,\infty)$, by definition of $\M_{\N}(0,\infty)$, there is $f\in C_c(0,\infty)$ such that $\int fd\mu =\infty$. We denote by $[\epsilon,C]$ the support of $f$ with $\epsilon,C>0$. Given that the following holds
\begin{align}\label{eqinfini}
    \int f d\mu \leq \frac{\|f\|_\infty}{\epsilon} \int x d\mu,
\end{align}
we deduce that $\int x d\mu=\infty$.
We then deduce from the continuity of $\nu \in \overline{M}_{\N}(0,\infty) \mapsto \int f d\nu$ that $\lim_{\delta \to 0}\inf_{\nu \in B_\delta(\mu)}\int fd\nu=\infty$ and thus from \ref{eqinfini},  we get the limit \eqref{eqmominf}.
\end{proof}

\begin{acks}
We acknowledge the support of Fondazione CR di Firenze  via  FABuLOuS project within ``Bando Ricercatori a Firenze'',  and the partial support of the INDAM GNAMPA Project, codice CUP\_E53C22001930001.
\end{acks}
%%%%%%%%%%%%%%%%%%%%%%%%%%%%%%%%%%%%%%%%%%%%%%%%%%%%%%%%%%%%%%%%%%%
%%                                                               %%
%% Supplementary Material, if any, should be provided in         %%
%% {supplement} environment  with title and short description.   %%
%%                                                               %%
%%%%%%%%%%%%%%%%%%%%%%%%%%%%%%%%%%%%%%%%%%%%%%%%%%%%%%%%%%%%%%%%%%%

% \begin{supplement}
% \stitle{Title of Supplement A.}
% \sdescription{Short description of Supplement A.}
% \end{supplement}
% \begin{supplement}
% \stitle{Title of Supplement B.}
% \sdescription{Short description of Supplement B.}
% \end{supplement}

%%%%%%%%%%%%%%%%%%%%%%%%%%%%%%%%%%%%%%%%%%%%%%%%%%%%%%%%%%%%%%%%%%%
%%                                                               %%
%% Use the two commands below for producing your bibliography    %%
%% with bibtex, then comment again the commands and include the  %%
%% content of the .bbl file in this file below the commands.     %%
%%                                                               %%
%%%%%%%%%%%%%%%%%%%%%%%%%%%%%%%%%%%%%%%%%%%%%%%%%%%%%%%%%%%%%%%%%%%

\bibliographystyle{amsplain}
\bibliography{biblio_graph}

\providecommand{\bysame}{\leavevmode\hbox to3em{\hrulefill}\thinspace}
\providecommand{\MR}{\relax\ifhmode\unskip\space\fi MR }
\providecommand{\MRhref}[2]{%
  \href{http://www.ams.org/mathscinet-getitem?mr=#1}{#2}
}
\providecommand{\href}[2]{#2}
\begin{thebibliography}{10}

\bibitem{AL97}
David Aldous, \emph{Brownian excursions, critical random graphs and the
  multiplicative coalescent}, Ann. Probab. \textbf{25} (1997), no.~2, 812--854.
  \MR{1434128}

\bibitem{AnKoLaPa23}
Luisa Andreis, Wolfgang K\"{o}nig, Heide Langhammer, and Robert I.~A.
  Patterson, \emph{A large-deviations principle for all the components in a
  sparse inhomogeneous random graph}, Probab. Theory Related Fields
  \textbf{186} (2023), no.~1-2, 521--620. \MR{4586227}

\bibitem{LA21}
Luisa Andreis, Wolfgang K\"{o}nig, and Robert I.~A. Patterson, \emph{A
  large-deviations principle for all the cluster sizes of a sparse
  {E}rd{\H{o}}s-{R}{\'e}nyi graph}, Random Structures Algorithms \textbf{59}
  (2021), no.~4, 522--553. \MR{4323309}

\bibitem{BaBoDe00}
Daniel Barraez, Stephane Boucheron, and W~Fernandez De~La~Vega, \emph{On the
  fluctuations of the giant component}, Combinatorics, Probability and
  Computing \textbf{9} (2000), no.~4, 287--304.

\bibitem{bdmckay}
Edward~A. Bender, E.~Rodney Canfield, and Brendan~D. McKay, \emph{The
  asymptotic number of labeled connected graphs with a given number of vertices
  and edges}, Random Structures Algorithms \textbf{1} (1990), no.~2, 127--169.
  \MR{1138421}

\bibitem{bet2020big}
Gianmarco Bet, Remco Van Der~Hofstad, and Johan~SH Van~Leeuwaarden, \emph{Big
  jobs arrive early: From critical queues to random graphs}, Stochastic Systems
  \textbf{10} (2020), no.~4, 310--334.

\bibitem{BhvdHoLe10}
Shankar Bhamidi, Remco Van Der~Hofstad, and Johan van Leeuwaarden,
  \emph{Scaling limits for critical inhomogeneous random graphs with finite
  third moments}, Electron. J. Probab. \textbf{15} (2010), 1682--1702.

\bibitem{BoJaRi07}
B{\'e}la Bollob{\'a}s, Svante Janson, and Oliver Riordan, \emph{The phase
  transition in inhomogeneous random graphs}, Random Structures $\&$ Algorithms
  \textbf{31} (2007), no.~1, 3--122.

\bibitem{BoCa15}
Charles Bordenave and Pietro Caputo, \emph{Large deviations of empirical
  neighborhood distribution in sparse random graphs}, Probability Theory and
  Related Fields \textbf{163} (2015), no.~1-2, 149--222.

\bibitem{ER60}
Paul {E}rd{\H{o}}s and Alfr\'{e}d R{\'e}nyi, \emph{On the evolution of random
  graphs}, Magyar Tud. Akad. Mat. Kutat\'{o} Int. K\"{o}zl. \textbf{5} (1960),
  17--61. \MR{125031}

\bibitem{ER61}
Paul {E}rd{\H{o}}s and Alfr\'{e}d R\'{e}nyi, \emph{On the evolution of random
  graphs}, Bull. Inst. Internat. Statist. \textbf{38} (1961), 343--347.
  \MR{148055}

\bibitem{Mar98}
Anders Martin-L{\"o}f, \emph{The final size of a nearly critical epidemic, and
  the first passage time of a wiener process to a parabolic barrier}, Journal
  of Applied probability \textbf{35} (1998), no.~3, 671--682.

\bibitem{OCon98}
Neil O'Connell, \emph{Some large deviation results for sparse random graphs},
  Probability theory and related fields \textbf{110} (1998), 277--285.

\bibitem{Pit90}
Boris Pittel, \emph{On tree census and the giant component in sparse random
  graphs}, Random Structures $\&$ Algorithms \textbf{1} (1990), no.~3,
  311--342.

\bibitem{Pu05}
Anatolii~A. Puhalskii, \emph{Stochastic processes in random graphs}, Ann.
  Probab. \textbf{33} (2005), no.~1, 337--412. \MR{2118868}

\bibitem{Sod03}
Bo~S{\"o}derberg, \emph{Random graphs with hidden color}, Physical Review E
  \textbf{68} (2003), no.~1, 015102.

\bibitem{Ste70}
Vadim~E. Stepanov, \emph{On the probability of connectedness of a random graph
  $g\_m(t)$}, Theory of Probability $\&$ Its Applications \textbf{15} (1970),
  no.~1, 55--67.

\bibitem{Sun23}
Wen Sun, \emph{A conditional compound poisson process approach to the sparse
  erd$\backslash$h $\{$o$\}$ sr$\backslash$'enyi random graphs: moderate
  deviations}, arXiv preprint arXiv:2310.06348 (2023).

\bibitem{vdHof16}
Remco Van Der~Hofstad, \emph{Random graphs and complex networks}, vol.~43,
  Cambridge university press, 2016.

\end{thebibliography}

\end{document}